\documentclass[12pt]{amsart}
\usepackage{amsfonts,amssymb,amsmath,amscd,amsthm, amstext,bm,color}
\usepackage[colorlinks, citecolor=blue,pagebackref,hypertexnames=false]{hyperref}
\usepackage{enumerate}
\usepackage{epstopdf}
\usepackage{fancyhdr}
\usepackage{geometry}
\usepackage{mathrsfs}
\usepackage{txfonts}
\usepackage{tikz}
\usepackage{url}
\usepackage{ulem}
\allowdisplaybreaks

\newtheorem{theorem}{Theorem}[section]

\newtheorem{lemma}[theorem]{Lemma}

\newtheorem{remark}[theorem]{Remark}

\begin{document}

\newcommand{\norm}[1]{\left\Vert#1\right\Vert}
\newcommand{\abs}[1]{\left\vert#1\right\vert}
\newcommand{\set}[1]{\left\{#1\right\}}
\newcommand{\Real}{\mathbb{R}}
\newcommand{\RR}{\mathbb{R}^n}
\newcommand{\supp}{\operatorname{supp}}
\newcommand{\card}{\operatorname{card}}
\renewcommand{\L}{\mathcal{L}}
\renewcommand{\P}{\mathcal{P}}
\newcommand{\T}{\mathcal{T}}
\newcommand{\A}{\mathcal{A}}
\newcommand{\K}{\mathcal{K}}
\renewcommand{\SS}{\mathcal{S}}
\newcommand{\blue}[1]{\textcolor{blue}{#1}}
\newcommand{\red}[1]{\textcolor{red}{#1}}
\newcommand{\Id}{\operatorname{I}}

\title[Endpoint boundedness of singular integrals]
{Endpoint boundedness of singular integrals: CMO space associated to Schr\"odinger operators}

\author{ Xueting Han}
\address{Xueting Han, School of Mathematics and Physics, University of Science and Technology Beijing,
Beijing 100083, P.R. China  and
    School of Mathematical and Physical Sciences, Macquarie University, NSW 2109, Australia}
\email{hanxueting12@163.com}

 \author{Ji Li}
\address{Ji Li, School of Mathematical and Physical Sciences, Macquarie University, NSW 2109, Australia}
\email{ji.li@mq.edu.au}

\author{Liangchuan Wu}
\address{Liangchuan Wu,  School of Mathematical Science, Anhui University, Hefei, 230601, P.R.~China}
\email{wuliangchuan@ahu.edu.cn}

 \date{\today}

\begin{abstract}  
Let $ \mathcal{L} = -\Delta + V $ be a Schrödinger operator acting on $ L^2(\mathbb{R}^n) $, where the nonnegative potential \( V \) belongs to the reverse Hölder class $ RH_q $ for some $ q \geq  n/2 $. This article is primarily concerned with the study of endpoint boundedness for classical singular integral operators in the context of the space $ \mathrm{CMO}_{\mathcal{L}}(\mathbb{R}^n) $, consisting of functions of vanishing mean oscillation associated with $ \mathcal{L} $.

We establish the following main results: (i) the standard Hardy--Littlewood maximal operator is bounded on $\mathrm{CMO}_{\mathcal{L}}(\mathbb{R}^n) $; (ii) for each $ j = 1, \ldots, n$, the adjoint of the Riesz transform $ \partial_j \mathcal{L}^{-1/2} $ is bounded from $ C_0(\mathbb{R}^n) $ into $ \mathrm{CMO}_{\mathcal{L}}(\mathbb{R}^n) $; and (iii) the approximation to the identity generated by the Poisson and heat semigroups associated with $ \mathcal{L} $ characterizes $ \mathrm{CMO}_{\mathcal{L}}(\mathbb{R}^n) $ appropriately.

These results recover the classical analogues corresponding to the Laplacian as a special case. However, the presence of the potential $ V $ introduces substantial analytical challenges, necessitating tools beyond the scope of classical Calderón--Zygmund theory. Our approach leverages precise heat kernel estimates and the structural properties of $ \mathrm{CMO}_{\mathcal{L}}(\mathbb{R}^n) $ established by Song and the third author in \cite{SW}.

 \end{abstract}

\maketitle

 \tableofcontents

%%%%%%%%%%%%%%%%%%%%%%%%%%%%%%%%%%%%%%%%%%%%%%%%%%%%%%%%%%%%%%%%%%%%%%%%%%%%%%%%%%%%%

 %                                   INTRODUCTION

%%%%%%%%%%%%%%%%%%%%%%%%%%%%%%%%%%%%%%%%%%%%%%%%%%%%%%%%%%%%%%%%%%%%%%%%%%%%%%%%%%%%%

\section{Introduction and main results}
\setcounter{equation}{0}

Let us consider  the Schr\"odinger  operator
$$ 
    \mathcal{L}=-\Delta+V(x) \   \    { \rm on } \  \  L^{2}(\mathbb{R}^{n}), \quad n \geq 3,
$$ 
where the nonnegative potential $V$ is not identically zero, and $V\in RH_q$  for  some $q\geq n/2$,
which by definition means that   $V\in L^{q}_{\rm loc}(\mathbb{R}^n), V\geq 0$, and
there exists a  constant $C>0$ such that  the reverse H\"older inequality
\begin{equation}\label{eqn:Reverse-Holder}
  \left(\frac{1}{|B|} \int_{B} V(y)^{q} d y\right)^{1 / q} \leq \frac{C}{|B|} \int_{B} V(y)\, d y
\end{equation}
holds for all   balls $B$ in $\mathbb{R}^n$. Following \cite{DGMTZ},   a locally integrable function $f$ belongs to  ${\rm BMO}_{{\mathcal{L}}}({\mathbb R}^n)$ if 
\begin{equation}\label{eqn:def-BMOL-norm}
     \|f\|_{{\rm BMO}_{\L}(\mathbb{R}^n)}:=\sup_{B=B(x_B,r_B):\, r_B<\rho(x_B)}\frac{1}{|B|} \int_{B}\left|f(y)-f_{B}\right|   d y+\sup_{B=B(x_B,r_B):\, r_B\geq \rho(x_B)} \frac{1}{|B|} \int_{B}|f(y)|\, d y <\infty.
\end{equation}
The critical radii above are determined by the function $\rho(x; V)=\rho(x)$, which was first introduced by Shen \cite[Definition~1.3]{Shen1} and  takes the explicit form
\begin{equation}\label{eqn:critical-funct}
 \rho(x)=\sup \left\{r>0: \frac{1}{r^{n-2}} \int_{B(x, r)} V(y)\, d y \leq 1\right\}.
\end{equation}

This article focuses on ${\rm CMO}_{\L}(\mathbb R^n)$, the space of vanishing mean oscillation associated to $\L$, which is the closure of  $C_c^\infty(\mathbb R^n)$ (the space of smooth functions with compact support) in the ${\rm BMO}_{\L}(\mathbb R^n)$ norm. As a crucial  subspace of ${\rm BMO}_{\L}(\mathbb R^n)$, it satisfies the duality relations 
\begin{equation}\label{eqn:duality}  
   \left ({\rm CMO}_{\L}(\mathbb R^n)\right)^*=H_{\L}^1(\mathbb R^n)\quad \text{and} \quad (H_{\L}^1(\mathbb R^n))^*={\rm BMO}_{\L}(\mathbb R^n),
\end{equation}
where the Hardy-type space $H_{\L}^1(\mathbb{R}^n)$  is defined by
$$
H_{\mathcal{L}}^{1}\left(\mathbb{R}^{n}\right)=\left\{f \in L^{1}\left(\mathbb{R}^{n}\right): \, \mathcal{P}^{*} f(x)=\sup _{t>0}\left|e^{-t \sqrt{\mathcal{L}}} f(x)\right| \in L^{1}(\mathbb{R}^{n})\right\}
$$ 
with norm 
$\|f\|_{H_{\mathcal{L}}^{1}\left(\mathbb{R}^{n}\right)}=\big\|\,\mathcal{P}^{*} f\,\big\|_{L^{1}\left(\mathbb{R}^{n}\right)}$.
See \cite{DDSTY, DGMTZ, Ky} for details. Additional equivalent characterizations of ${\rm CMO}_{\L}(\mathbb R^n)$ via mean oscillation and   tent spaces, respectively, can be found in  \cite{SW} by L. Song and the third author.

The space ${\rm CMO}_{\L}(\mathbb R^n)$ shares key similarities with the classical vanishing mean oscillation space: when $V\equiv 0$, ${\rm CMO}_{\Delta}(\mathbb R^n)$ (resp. ${\rm BMO}_{\Delta}(\mathbb R^n)$) coincides exactly with the standard ${\rm CMO}(\mathbb R^n)$ (resp. ${\rm BMO}(\mathbb R^n)$), and the  dualities \eqref{eqn:duality}  reduce to their  classical counterparts. However, ${\rm CMO}_{\L}(\mathbb R^n)$ demonstrates certain properties distinct from the classical setting. For instance,  the convolution of  a compactly supported bump function with a function of  ${\rm CMO}_{\L}(\mathbb R^n)$ may fail to remain in 
 ${\rm CMO}_{\L}(\mathbb R^n)$; see \cite[Lemma 4.1]{SW}.

The aim of this paper is to study endpoint boundedness for classical singular integral operators in the context of the space $ \mathrm{CMO}_{\mathcal{L}}(\mathbb{R}^n)$.
Central to this pursuit are the boundedness of cornerstone operators such as the Hardy--Littlewood maximal operator and the Riesz transforms on this space. Additionally, the development of suitable approximations to the identity compatible with the structure of this space requires careful consideration.

{\bf Part I.} \ \ 
The (uncentered) Hardy--Littlewood maximal function $M$ on $\mathbb R^n$ is a well-known operator and plays a fundamental role in harmonic analysis.
However, its behaviour on the classical ${\rm CMO}(\mathbb R^n)$ remains unclarified.

Recall that for the classical case with $V\equiv 0$, it's known that for a function $f\in {\rm BMO}(\mathbb R^n)$, it may occur that $Mf\equiv +\infty$, and a typical example is $f(x)=\log |x|$ (in contrast, for any $f\in {\rm BMO}_{\L}(\mathbb R^n)$,  we have $Mf(x)<+\infty$ for a.e. $x\in \mathbb R^n$).  Nevertheless, there exists a constant $C$ depending only on $n$ such that for any $f\in {\rm BMO}(\mathbb R^n)$ for which $Mf$ is not identically equal to infinity, we have 
$$ 
   \|Mf\|_{{\rm BMO}(\mathbb R^n)}\leq C\|f\|_{{\rm BMO}(\mathbb R^n)};
$$ 
see \cite[Theorem 4.2]{BDS} by Bennett, DeVore and Sharpley. The further boundedness of $M$ and its  fractional counterpart
 on ${\rm VMO}$ were investigated in \cite{S} and \cite{GK}, respectively, where ${\rm VMO}$ is the BMO-closure of ${\rm UC}\cap {\rm BMO}$, and ${\rm UC}$ is the class of all uniformly continuous functions.  Alternatively,
$f\in {\rm VMO}(\mathbb R^n)$ if and only if $f\in {\rm BMO}(\mathbb R^n)$ and  
$$
 \lim _{a \rightarrow 0} \sup _{B: \,r_{B} \leq a} |B|^{-1} \int_{B}\left|f(x)-f_{B}\right| d x =0.
$$
Since the nonnegative potential $V$ is assumed  not to be  identically  zero, we have 
$$
     {\rm CMO}_{\L}(\mathbb R^n)\subsetneqq {\rm CMO}(\mathbb R^n)\subsetneqq  {\rm VMO}(\mathbb R^n).
$$

Our first result is  to characterize the Hardy--Littlewood maximal  operator $M$ on  $ {\rm CMO}_{\L}(\mathbb R^n)$, which also  clarifies the boundedness of $M$ on the classical $ {\rm CMO}(\mathbb R^n)$.

\begin{theorem}\label{thm:M-CMO}
Suppose $V\in RH_q$ for some $q\geq n/2$ and  let $\L=-\Delta+V$. 
For each $f\in {\rm BMO}_{\L}(\mathbb R^n)$, the Hardy--Littlewood maximal function $Mf$ belongs to ${\rm BMO}_{\L}(\mathbb R^n)$ as well, with 
\begin{equation}\label{eqn:M-BMO}
   \|Mf\|_{{\rm BMO}_{\L}(\mathbb R^n)}\leq C\|f\|_{{\rm BMO}_{\L}(\mathbb R^n)},
\end{equation}
where the constant $C>0$ is independent of $f$.

Moreover, if $f\in  {\rm CMO}_{\L}(\mathbb R^n)$, we also have $Mf\in {\rm CMO}_{\L}(\mathbb R^n)$.
\end{theorem}

To prove it, we will apply the characterization of ${\rm CMO}_{\L}(\mathbb R^n)$ in terms of the behaviour of mean osicllation given in \cite{SW} (see (vi) of Theorem \ref{thm:CMO-dual-known} below), and give a more refined modification of the argument for \cite[Theorem 4.2]{BDS}.  Note that ${\rm CMO}(\mathbb R^n)$ can  also be characterized via mean oscillation, which coincides with ${\rm CMO}_{\L}(\mathbb R^n)$ whenever taking $V\equiv 0$, hence 
our proof also reveals the behaviour of $M$ on the classical ${\rm CMO}(\mathbb R^n)$ (Remarkably,  functions $f\in {\rm CMO}(\mathbb R^n)$ for which $Mf$ is identically infinite must be ruled out, such as $f(x)=\ln \ln |x| \cdot \mathsf 1 _{\{|\cdot|\geq e\}}(x)$). See Remark \ref{rem:CMO-M} for details.

{\bf Part II.} \ \ Consider the $j$\,th Riesz transform $\displaystyle R_j=\frac{\partial}{\partial x_j} \L^{-1/2}$ associated to $\L$ on $\mathbb R^n$, $j=1,\ldots, n$.    Shen \cite{S}  established that when $V\in RH_q$ for $n/2\leq q<n$, then
$$
    \|R_j f\|_{L^p(\mathbb R^n)}\leq C_p \|f\|_{L^p(\mathbb R^n)}  \quad \text{for}\   \  1<p\leq p_0,
$$
where $\displaystyle \frac{1}{p_0}=\frac{1}{q}-\frac{1}{n}$. When $V\in RH_n$, $R_j$ is a Calder\'on--Zygmund operator for each $j$.
Hence it suffices to consider the case $V\in RH_q$ with $n/2
\leq q<n$.
Let $R_j(x,y)$ be the kernel of the Riesz transform $ R_j$. Then the adjoint of $R_j$ is given by 
 $$
    R_j^* g(x)=\lim_{\varepsilon\to 0}\int_{|y-x|>\varepsilon}  R_j(y,x)g(y) dy.
 $$
By duality,  the above boundedness of $R_j$ deduces that $R_j^*$ is bounded on $L^{p'}(\mathbb R^n)$ with $p_0'\leq p'<\infty$, where $\displaystyle \frac{1}{p}+\frac{1}{p'}=1$.  Moreover, $R_j^*$ is bounded from $L^\infty(\mathbb R^n)$ to ${\rm BMO}_{\L}(\mathbb R^n)$, which is useful to give a characterization of ${\rm BMO}_{\L}(\mathbb R^n)$ via $R_j^*$. Concretely, for each $f\in {\rm BMO}_{\L}(\mathbb R^n)$,  we can write
$$
   f=\phi_0+\sum_{j=1}^nR_j^* \phi_j,\quad \phi_j\in L^\infty(\mathbb R^n), \, 0\leq j\leq n.
$$
See \cite[Theorem 1.3]{WY} by the third author and L.X. Yan.

\smallskip
To continue this line, our second result is as follows. Let  $C_0(\mathbb{R}^n)$ be the space of all continuous functions on $\mathbb{R}^n$ which vanish at infinity.

\begin{theorem}\label{thm:Riesz-CMO}
Suppose $V\in RH_q$ for some $q\geq n/2$ and  let $\L=-\Delta+V$. 
The adjoint Riesz transform  $R_j^*$ associated to $\L$ is bounded from $C_0(\mathbb{R}^n )$ to ${\rm CMO}_{\L}(\mathbb{R}^n)$  for $j=1,\ldots, n$.
\end{theorem}

When $V\equiv 0$, the boundedness above is known, based on the Fourier transform of the classical Riesz transform $\displaystyle  \frac{\partial}{\partial x_j} (-\Delta)^{-1/2}$ for  $j=1,\ldots, n$;  see   \cite[Lemma~1]{D} for details. For any generic potential $V\in RH_q $,  techniques from Fourier transform are not workable, and we will show Theorem~\ref{thm:Riesz-CMO} by exploiting estimates for the kernels of Riesz transforms  and applying preliminaries in \cite{Shen1}.

As a consequence, we will show (see Lemma \ref{lem:Riesz-CMO-2} below) a  Riesz-type representation that  
for every continuous linear functional $\ell$ on  ${\rm CMO}_{\L}(\mathbb{R}^n)$, there exists a uniquely finite Borel measure $\mu_0$ 
 such that  $\ell$ can be realized by 
$$
     \ell(g)=\int_{\mathbb{R}^n}   g(x)\, d\mu_0(x),\quad \forall\, g\in {\rm CMO}_{\L}(\mathbb{R}^n),
$$
where $\mu_0$ satifies that its Riesz transforms $R_j(d \mu_0)(x)=\int R_j (x,y)\,d\mu_0(y)$ associated to $\L$ for $j=1,2,\ldots, n$, are all finite Borel measures.

\smallskip

{\bf Part III.} Consider the approximation to the identity on ${\rm CMO}_{\L}(\mathbb R^n)$. As aforementioned, the standard approximation to the identity can not match ${\rm CMO}_{\L}(\mathbb R^n)$ well due to the potential $V$. 
 Even for a radial bump function $\phi$ satisfying
$$ 
   {\rm supp}\, \phi\subseteq B(0,1),\quad        0\leq \phi\leq 1  \quad {\rm and }\quad \int \phi(x)\, dx=1,
$$ 
 the convolution $A_tf=t^{-n}\phi(t^{-1}\cdot)*f $ for $f\in {\rm CMO}_{\L}(\mathbb R^n)$ may not belong to ${\rm CMO}_{\L}(\mathbb R^n)$, unless assuming additional conditions such as  $f\in C_c^\infty(\mathbb R^n)$. In this article, we consider   the approximation to the identity arising from semigroups associated to $\L$.

\begin{theorem}\label{thm:approx}
Suppose $V\in RH_q$ for some $q\geq n/2$ and  let $\L=-\Delta+V$.  For any $f\in {\rm CMO}_{\L}(\mathbb R^n)$, we have $e^{-t\sqrt{\L}}f\in {\rm CMO}_{\L}(\mathbb R^n)$ for each $t>0$, and 
\begin{equation}\label{eqn:approx-identity-CMO}
   \lim_{t\to 0} e^{-t\sqrt{\L}}f=f\quad {\rm in }\  \   {\rm BMO}_{\L}(\mathbb R^n).
\end{equation}
In particular, if $f\in C_c^\infty(\mathbb R^n)$, then  we also have $\displaystyle \lim_{t\to 0} e^{-t\sqrt{\L}}f(x)=f(x)$ uniformly for all $x\in \mathbb R^n$.
\end{theorem}

The analogous conclusion remains valid when replacing the Poisson semigroup $e^{-t\sqrt{\L}}$ by the heat semigroup $e^{-t\L}$.

To establish this, we first show that  for any $f\in {\rm BMO}_{\L}(\mathbb R^n)$ and $t>0$,  the function $e^{-t\sqrt{\L}}f$ also  belongs to $ {\rm BMO}_{\L}(\mathbb R^n)$, and a corresponding result holds in  ${\rm CMO}_{\L}(\mathbb R^n)$ (see Lemma  \ref{lem:Poisson-CMO}). Our main ingredient is the characterization of $ {\rm CMO}_{\L}(\mathbb R^n)$ via the theory of tent spaces established in \cite{SW}.  Consequently,
it suffices to verify \eqref{eqn:approx-identity-CMO}  for functions in $C_c^\infty(\mathbb R^n)$, and we can utilize the classical Poisson semigroup $e^{-t\sqrt{-\Delta}}$ to streamline the argument.

This paper is organized as follows. 
In Section \ref{Preliminaries} we introduce the necessary preliminaries in characterizations of ${\rm CMO}_{\L}(\mathbb R^n)$ and the auxiliary function  $\rho$.
In Section \ref{sec:HL} we provide the proof of Theorem  \ref{thm:M-CMO}. 
In Section \ref{sec:Riesz} we present the proof of Theorem \ref{thm:Riesz-CMO}. 
The argument for Theorem \ref{thm:approx} will be discussed in the last section.

\section{Preliminaries}\label{Preliminaries}
\setcounter{equation}{0}

We recall some preliminaries on ${\rm CMO}_{\L}(\mathbb R^n)$ and the auxiliary function $\rho$ defined in \eqref{eqn:critical-funct}.

A remarkable fact is the self-improvement property: if $V\in RH_q$ with $q>1$, then there exists $\varepsilon>0$ depending only on the constant $C$ in \eqref{eqn:Reverse-Holder} and the dimension $n$ such that $V\in RH_{q+\varepsilon}$. Consequently, the assumption ``$V\in RH_q$ for some $q\geq n/2$''    can be rewritten as ``$V\in RH_q$ for some $q> n/2$''. This fact is useful for dealing with some critical indices that appear in our article below.

\smallskip

Combining works in \cite{DDSTY, Ky, SW}, we have the following characterizations of ${\rm CMO}_{\L}(\mathbb{R}^n)$.

\begin{theorem}\label{thm:CMO-dual-known}
Suppose $V\in RH_q$ for some $q\geq n/2$ and  let $\L=-\Delta+V$. The following statements are equivalent.
\begin{itemize}
	\item[(i) ] $f$ is in  ${\rm CMO}_{\L}(\mathbb{R}^n)$.

		\smallskip
	
	\item[(ii) ]  $f$ is in the closure in the ${\rm BMO}_{\L}(\mathbb{R}^n)$ norm of $C_c^\infty(\mathbb{R}^n)$.

	\smallskip

	\item[(iii) ]  $f$ is in the closure in the ${\rm BMO}_{\L}(\mathbb{R}^n)$ norm of $C_0(\mathbb{R}^n)$. 
	
	\smallskip 
    \item[(iv) ]	$f$ is in the pre-dual space of the Hardy space $H_{\L}^1(\mathbb{R}^n)$.
    
    	\smallskip
	
    \item [(v) ] $f$ is in $\mathcal{B}_{\L}$, where $\mathcal{B}_{\L}$ is the subspace of ${\rm BMO}_{\L}(\mathbb{R}^n)$ satisfying $\widetilde{\gamma}_i(f)=0$ for $1\leq i\leq 5$, where
\begin{align*}
\widetilde{\gamma}_1(f) &=\lim _{a \rightarrow 0} \sup _{B: \,r_{B} \leq a}\left(|B|^{-1} \int_{B}\left|f(x)-f_{B}\right|^{2} d x\right)^{1 / 2} ;\\
\widetilde{\gamma}_2(f) &=\lim _{a \rightarrow \infty} \sup _{B: \,r_{B} \geq a}\left(|B|^{-1} \int_{B}\left|f(x) -f_B\right|^{2} d x\right)^{1 / 2} ;\\
\widetilde{\gamma}_3(f) &=  \lim _{a \rightarrow \infty} \sup _{B: \, B \subseteq (B(0, a))^c}\left(|B|^{-1} \int_{B}\left|f(x)-f_{B}\right|^{2} d x\right)^{1 / 2};\\
\widetilde{\gamma}_4(f) &=\lim _{a \rightarrow \infty} \sup _{B: \,r_{B} \geq \max\{a,\,\rho(x_B)\}}\left(|B|^{-1} \int_{B}\left|f(x) \right|^{2} d x\right)^{1 / 2} ;\\
\widetilde{\gamma}_5(f) &=  \lim _{a \rightarrow \infty} \sup _{\substack{B: \, B \subseteq (B(0, a))^{c}\\   r_B\geq \rho(x_B)}}\left(|B|^{-1} \int_{B}\left|f(x)\right|^{2} d x\right)^{1 / 2}.
\end{align*}
Here $x_B$ denotes the center of $B$,  and  the function $\rho$ is defined in \eqref{eqn:critical-funct}.

	\smallskip

\item [(vi)] $f$ is in ${\rm BMO}_{\L}(\mathbb{R}^n)$  and satisfies $\widetilde{\gamma}_1(f)=\widetilde{\gamma}_3(f)=\widetilde{\gamma}_5(f)=0$.

\end{itemize}
\end{theorem}

\smallskip

Next we review the slowly varying  property of the critical radii function $\rho(x)$.

\begin{lemma} \label{lem:size-rho}
{\rm ( \cite[Lemma~1.4]{Shen1}.)}\
Suppose $V\in {\rm RH}_q$ for some $q\geq  n/2$. There exist $c>1$ and $k_0\geq1$ such that for all $x,y\in\mathbb{R}^n$,
\begin{equation}\label{eqn:size-rho}
c^{-1}\left(1+\frac{\abs{x-y}}{\rho(x)}\right)^{-k_0} \rho(x)\leq\rho(y)\leq c
\left(1+\frac{\abs{x-y}}{\rho(x)}\right)^{\frac{k_0}{k_0+1}} \rho(x).
\end{equation}
In particular, $\rho(x)\approx \rho(y)$ when $y\in B(x, r)$ and $r\lesssim\rho(x)$.
\end{lemma}
Hence $0<\rho(x)<\infty$ for each $x\in \mathbb R^n$, and $\rho$ is locally bounded from above and below. This fact will be used frequently in our article.

\medskip

\section{Hardy-Littlewood maximal operator on ${\rm CMO}_{\L}(\mathbb R^n)$: proof of Theorem  \ref{thm:M-CMO}} 
\label{sec:HL}
\setcounter{equation}{0}

Let $M$ denote  the uncentered Hardy--Littlewood maximal function.  The aim of this section is to explore the boundedness of $M$ on ${\rm CMO}_{\L}(\mathbb R^n)$.

\begin{proof}[Proof of Theorem~\ref{thm:M-CMO}]
{\it Step I.}
We begin by showing that for any given $f\in {\rm BMO}_{\L}(\mathbb R^n)$, we have  $Mf< +\infty$ for a.e. $x\in \mathbb R^n$.

This fact has been proven in \cite{DGMTZ} by splitting the function $f$ into a local part and a nonlocal part.
Alternatively, here we present an alternative proof by directly applying the definition  \eqref{eqn:def-BMOL-norm} of the ${\rm BMO}_{\L}$ norm.

 Indeed,   for any $f\in {\rm BMO}_{\L}$, it follows from the definition of the ${\rm BMO}_{\L}$ norm that 
 \begin{align*}
   Mf(x)\leq C_n  \sup_{r>0} \frac{1}{|B(x,r)|}\int_{B(x,r)} |f(y)|dy \leq C_n   \|f\|_{{\rm BMO}_{\L}} \sup_{r>0}   \max\left\{ \Big(\frac{\rho(x)}{r}\Big)^n, 1 \right\},
\end{align*}
and the $ \sup_{r>0}$ can be improved to $  \sup_{r>\delta}$ for some $\delta>0$ due to  the Lebesgue differentiation theorem. Specifically, for any $\varepsilon>0$, there exists $\delta=\delta(\varepsilon, x)>0$ such that 
$$
      Mf(x)\leq    \max\left\{  |f(x)|+\varepsilon, C_n   \|f\|_{{\rm BMO}_{\L}} \sup_{r>\delta}   \max\left\{ \Big(\frac{\rho(x)}{r}\Big)^n, 1 \right\}  \right\},
$$

Since
$0<\rho(x)<\infty$ for each $x\in \mathbb R^n$, we obtain $Mf(x)<+\infty$, a.e. $x\in \mathbb R^n$.

\smallskip

{\it Step II.} Next, we prove that $M$ is bounded on ${\rm BMO}_{\L}(\mathbb R^n)$.

Due to ${\rm BMO}_{\L}\subset {\rm BMO}$ and $\|f\|_{{\rm BMO}}\leq 2\|f\|_{{\rm BMO}_{\L}}$ for any $f\in {\rm BMO}_{\L}$, it follows from the boundedness of $M$ on the classical ${\rm BMO}$ space (see \cite[Theorem 4.2]{BDS}) that 
 $$
     \sup_{B=B(x_B,r_B):\, r_B<\rho(x_B)}\frac{1}{|B|} \int_{B}\left|Mf(y)-(Mf)_{B}\right|   d y\leq \|Mf\|_{{\rm BMO}}\leq C\|f\|_{{\rm BMO}}\leq C\|f\|_{{\rm BMO}_{\L}},
 $$
 and it suffices to show that 
\begin{equation}\label{eqn:aim1}
    \sup_{B=B(x_B,r_B):\, r_B\geq \rho(x_B)} \frac{1}{|B|} \int_{B}|Mf(y)|\, d y \leq C \|f\|_{{\rm BMO}_{\L}}.
\end{equation}
 
Recall that for every cube $Q$,
$$
    \frac{1}{|Q|}\int_Q Mf(x)dx\leq c\left( \|f\|_{{\rm BMO}}+ \inf_{x\in Q} Mf(x)\right),
$$
as shown in \cite[page 610]{BDS}. Therefore, for any given ball $B=B(x_B,r_B)$ with $r_B\geq \rho(x_B)$, there eixsts a cube $Q$ whose sigdelength is  $2r_B$ and $B\subseteq Q$, so
$$
     \frac{1}{|B|} \int_{B}|Mf(y)|\, d y\leq  c\left( \|f\|_{{\rm BMO}}+ \inf_{x\in B} Mf(x)\right).
$$
This implies that it suffices to consider \eqref{eqn:aim1} for balls $B$ with $ r_B= \rho(x_B)$. Observe that for each $x\in B(x_B,\rho(x_B))$,
\begin{align*}
    Mf(x)&=\max\left\{\sup_{B'\ni x,\, r_{B'}\leq \rho(x)}  \frac{1}{B'}  \int_{B'} |f(y)|d y,\,     \sup_{B'\ni x,\, r_{B'}>\rho(x)}  \frac{1}{B'}  \int_{B'} |f(y)|d y \right\} \\
    &=:\max\left\{ F_1 (x), \, F_2(x)\right\}.
\end{align*}

If $Mf(x_0)=F_2(x_0)$ for some $ x_0\in B(x_B,\rho(x_B))$, then 
$$
  \inf_{x\in B} Mf(x)\leq F_2 (x_0)\leq \|f\|_{{\rm BMO}_{\L}}.
$$

Otherwise, $Mf(x)=F_1(x)$ for each $ x\in B(x_B,\rho(x_B))$. Note that $\rho(x)\approx \rho(x_B)$ for each $x\in B(x_B,\rho(x_B))$ by Lemma \ref{lem:size-rho}, then $Mf(x)=M(f\mathsf 1_{\mathsf k B(x_B,\,\rho(x_B)})$ for some constant $\mathsf k>1$ independent of $B(x_B,\,\rho(x_B))$.
Thus by the $L^2$ boundedness of the operator $M$,
\begin{align*}
   \frac{1}{|B(x_B, \rho(x_B))|} \int_{B}|Mf(y)|\, d y &\leq \bigg( \frac{1}{|B(x_B, \rho(x_B))|} \int_{B}|Mf(y)|^2 dy\bigg)^{1/2}\\
   &\leq C \bigg(  \frac{1}{|B(x_B, \rho(x_B))|} \int_{\mathbb R^n}  \left|f(y)\mathsf 1_{\mathsf k B(x_B,\rho(x_B))}(y)\right|^2 dy\bigg)^{1/2} \\
   &  \leq C \|f\|_{{\rm BMO}_{\L}},
\end{align*}
where the last inequality above follows from the John--Nirenberg type inequality associated to ${\rm BMO}_{\L}$ (see  \cite[Corollary 3]{DGMTZ} for example). Hence, we complete the proof of \eqref{eqn:M-BMO}.

\smallskip

{\it Step III.} It remains to show that $Mf$ belongs to ${\rm CMO}_{\L}(\mathbb R^n)$ whenever $f\in {\rm CMO}_{\L}(\mathbb R^n)$.

For any given $f\in {\rm CMO}_{\L}(\mathbb R^n)$, it follows from Theorem \ref{thm:CMO-dual-known} that it's equivalent to  $f\in {\rm BMO}_{\L}(\mathbb R^n)$ and $\widetilde{\gamma}_1(f)=\widetilde{\gamma}_3(f)=\widetilde{\gamma}_5(f)=0$. By {\it Step II}, we have $Mf\in {\rm BMO}_{\L}(\mathbb R^n)$, hence
it suffices to verify that $\widetilde{\gamma}_1(Mf)=\widetilde{\gamma}_3(Mf)=\widetilde{\gamma}_5(Mf)=0$. 

The following argument refines and modifies the approach in \cite[Theorem 4.2]{BDS}; see also \cite{S}. 

Note that for any $B$,
\begin{align}\label{eqn:osc}
   \frac{1}{|B|} \int_B \left|f(x)-f_B\right|^2  dx& \leq \frac{1}{|B|^2 }\iint_{B\times B} \left|f(x)-f(y)\right|^2 dydx \nonumber\\
   &= \frac{1}{|B|^2 }\iint_{B\times B} \left|f(x)-f_B +  f_B- f(y)\right|^2 dydx \nonumber\\
   &\leq  \frac{2}{|B|} \int_B \left|f(x)-f_B\right|^2  dx,
\end{align}
and 
$$
     \frac{1}{|B|^2 }\iint_{B\times B} \big ||f(x)|-|f(y)|\big |^2 dydx \leq \frac{1}{|B|^2 }\iint_{B\times B} \left|f(x)-f(y)\right|^2 dydx,
$$
from which it follows readily that 
 $|f|\in {\rm CMO}_{\L}(\mathbb R^n)$.  Besides, $M|f|=Mf$. Thus we may assume without loss of generality that $f$ is nonnegative.

Now, for any  $0\leq f\in {\rm CMO}_{\L}(\mathbb R^n)$, note that  $Mf$ and  $\widetilde{\gamma}_i(Mf)$ for $i=1,3,5$ can be defined using cubes with sides parallel to the coordinate axes instead of balls.
 In the following proof, ``cube'' always refers to such cubes. For any given cube $Q$,  denote its sidelength by $\ell(Q)$. 
 Let $\kappa>0$ be a constant to be chosen later. For each $x\in Q$, define
$$
    M_1f(x):=\sup_{Q' \ni  x:\, \ell(Q')< \kappa \ell(Q)} f_{Q'}\quad 
    \text{and} \quad M_2f(x):=\sup_{Q' \ni  x:\, \ell(Q')\geq  \kappa \ell(Q)} f_{Q'}
$$
Clearly, $Mf=\max\big\{ M_1 f, M_2 f\big\}$ on $Q$. Set
$$
   \Omega=\left\{x\in Q:\, Mf(x)>(Mf)_Q\right\},\quad \Omega_1=\left\{x\in \Omega:\, M_1f(x)\geq M_2f(x)\right\}
$$
and  $ \Omega_2=\Omega\setminus \Omega_1$. Then
\begin{align}\label{eqn:M-decomp}
   \frac{1}{|Q|}\int_{Q} \left|   Mf(x)-(Mf)_Q  \right| dx &=\frac{2}{|Q|}\int_{\Omega}  \Big(  Mf(x)-(Mf)_Q \Big ) \,dx\nonumber\\
   &=2 \sum_{i=1}^2  \frac{1}{|Q|} \int_{\Omega_i }  \Big (   M_if(x)-(Mf)_Q  \Big ) \, dx.
\end{align}

We begin by considering the term involving $M_1$ in \eqref{eqn:M-decomp}.  For the above cube $Q$, let  $Q^*= (2\kappa+1)Q $ denote the cube with the same center as $Q$ and sidelength $(2\kappa+1)\ell(Q)$. Then 
$M_1f(x)=M_1(f\mathsf 1_{Q^*})(x)$ for any $x\in Q$, and  
\begin{align}\label{eqn:M1}
\frac{1}{|Q|}\int_{\Omega_1 }  \Big (   M_1f(x)-(M f)_Q  \Big ) \, dx&\leq 
	\frac{1}{|Q|}\int_{\Omega_1 }  \Big (   M_1f(x)-(M_1 f)_Q  \Big ) \, dx\nonumber\\
	&\leq  \frac{1}{|Q|}\int_{Q} \left| M_1f(x)- f_{Q^*}  + f_{Q^*}- (M_1f)_Q \right|dx  \nonumber\\
	&\leq \frac{2}{|Q|}\int_{Q} \left| M_1f(x)- f_{Q^*}   \right| dx\nonumber\\
	&\leq 2\,\bigg(\frac{1}{|Q|}\int_{Q}\left| M_1 \left( f\mathsf 1_{Q^*}- f_{Q^*} \right)(x)\right|^2 dx\bigg)^{1/2}  \nonumber\\
	&\leq 2\,\bigg(\frac{1}{|Q|}\int_{Q}\left| M \left[ \left(f- f_{Q^*}\right) \mathsf 1_{Q^*} \right](x)\right|^2 dx\bigg)^{1/2} \nonumber\\
	&\leq C \kappa^{n/2} \bigg(\frac{1}{|Q^*|}\int_{Q^*} \left|f(x)-f_{Q^*}\right|^2 dx\bigg)^{1/2}.
\end{align}

It remains to consider the other term (i.e., $i=2$) on the right hand of \eqref{eqn:M-decomp}.  For any fixed $x\in \Omega_2$, we have $Mf(x)=M_2f(x)> (Mf)_Q$. Let $Q'$ be any cube containing $x$ with $\ell(Q')\geq \kappa \ell(Q)$. Let $Q''$ be a cube with $\ell(Q'')= \ell(Q)+\ell(Q')$ which contains both $Q$ and $Q'$ and shares some common faces (i.e., they have a common vertex); see Figure \ref{fig:1}. 

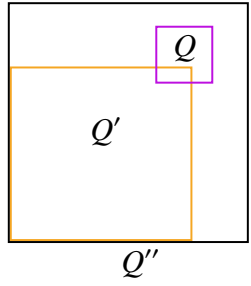
\begin{figure}

\tikzset{every picture/.style={line width=0.75pt}}  
\begin{tikzpicture}[x=0.75pt,y=0.75pt,yscale=-1,xscale=1]
 
\draw  [color={rgb, 255:red, 245; green, 166; blue, 35 }  ,draw opacity=1 ] (71.11,81.78) -- (162.22,81.78) -- (162.22,168.89) -- (71.11,168.89) -- cycle ;
 \draw  [color={rgb, 255:red, 189; green, 16; blue, 224 }  ,draw opacity=1 ] (144.5,61.21) -- (172.74,61.21) -- (172.74,89.44) -- (144.5,89.44) -- cycle ;
 
\draw   (70.11,49.33) -- (190.67,49.33) -- (190.67,169.89) -- (70.11,169.89) -- cycle ;

 \draw (172.22,72.11) node   [align=left] {\begin{minipage}[lt]{29.16pt}\setlength\topsep{0pt}
$\displaystyle Q$
\end{minipage}};
 
\draw (109.56,105.11) node [anchor=north west][inner sep=0.75pt]   [align=left] {$\displaystyle Q'$};
 
\draw (125.11,171.67) node [anchor=north west][inner sep=0.75pt]   [align=left] {$\displaystyle Q''$};

\end{tikzpicture}
\caption{$Q''$ contains $Q$ and $Q'$, and shares  a common vertex with $Q'$.}
  \label{fig:1}
\end{figure}

Then $Mf(y)\geq f_{Q''}$ for any $y\in Q$, so $f_{Q''}\leq (Mf)_{Q}$. Without loss of generality, one may assume that $\ell(Q')/\ell(Q)\in \mathbb N_+$,  thus  $Q^{\prime \prime} \backslash Q^{\prime}$ can be partitioned into $\frac{|Q''|-|Q'|}{|Q|}=(1+\frac{\ell(Q')}{\ell(Q)})^n-(\frac{\ell(Q')}{\ell(Q)})^n$ mutually disjoint cubes of side length $\ell(Q)$. Hence, by the pigeonhole principle,
there exists a cube $P\subset Q''\setminus Q$ with $\ell(P)=\ell(Q)$ such that $f_P\leq f_{Q''\setminus Q'}$.
As a consequence and by the nonnegative assumption of $f$,
\begin{align*}
   f_{Q'}-(Mf)_{Q}&\leq f_{Q'}-f_{Q''}\\
   &=f_{Q'}-\frac{|Q'|}{|Q''|} f_{Q'} -\frac{1}{|Q''|}\int_{Q''\setminus {Q'}} f(y)dy\\
   &\leq \frac{|Q''|-|Q'|}{|Q''|} \left[f_{Q'}-f_{Q''\setminus {Q'}}\right]\\
   &  \leq \frac{|Q''|-|Q'|}{|Q''|} \left[f_{Q'}-f_{P}\right]\\
   &\lesssim \frac{(\ell(Q')+\ell(Q))^n-\ell(Q')^n}{(\ell(Q')+\ell(Q))^n} \log\left(\frac{\ell(Q')}{\ell(Q)}\right) \|f\|_{{\rm BMO}}\\
   &\lesssim \frac{\log \kappa}{\kappa } \|f\|_{{\rm BMO}}.
\end{align*}
Taking the supremum over all such cubes $Q'$, we obtain 
$$
    M_2f(x)-(Mf)_Q\leq C \frac{\log \kappa}{\kappa } \|f\|_{{\rm BMO}}\quad \text{for}\   \  x\in \Omega_2.
$$
Consequently,
$$
   \frac{1}{|Q|}\int_{\Omega_2 }  \Big (   M_2f(x)-(M f)_Q  \Big ) \, dx\leq C \frac{\log \kappa}{\kappa } \|f\|_{{\rm BMO}}.
$$

Therefore, for any $\kappa>1$,
\begin{equation}\label{eqn:aim11}
\frac{1}{|Q|}\int_{Q} \left|   Mf(x)-(Mf)_Q  \right| dx \leq C\left[\kappa^{n/2} \bigg(\frac{1}{|Q^*|}\int_{Q^*} \left|f(x)-f_{Q^*}\right|^2 dx\bigg)^{1/2}  + \frac{\log \kappa}{\kappa } \|f\|_{{\rm BMO}} \right].
\end{equation}

Now we verify $\widetilde{\gamma}_i(Mf)=0$ for $i=1,3,5$. For any given  $\varepsilon>0$,  since $f\in {\rm CMO}_{\L}(\mathbb R^n)$ satisfies   $\widetilde{\gamma}_i(f)=0$ for $i=1,3,5$,  
there exist two integers $I_\varepsilon >>1$ and $J_\varepsilon>>1$ such that

 \begin{subequations}
\begin{equation}\label{eqn:CMO-a}
   \sup_{P:\, \ell(P)\leq 2^{-I_\varepsilon}} \left(\frac{1}{|P|}\int_P |f(x)-f_P|^2 dx \right)^{1/2}< \varepsilon,
\end{equation}
\begin{equation}\label{eqn:CMO-b}
    \sup_{P:\, P\subseteq (Q(0, 2^{J_\varepsilon }))^c} \left(\frac{1}{|P|}\int_P |f(x)-f_P|^2 dx\right)^{1/2}< \varepsilon,
\end{equation}
\begin{equation}\label{eqn:CMO-c}
    \sup_{ P\subseteq (Q(0, 2^{J_\varepsilon}))^c,\, \ell(P)\geq \rho(c_P)} \left(\frac{1}{|P|}\int_P |f(x) |^2 dx\right)^{1/2}< \varepsilon,
\end{equation}
 \end{subequations}
where $c_P$ dentoes the center of  the cube $P$, and $Q(0, 2^{J_\varepsilon })$ denotes the cube with $c_Q=0$ and $\ell(Q)=2^{J_\varepsilon }$.  

Hence,  taking $\displaystyle \kappa= \varepsilon^{-1/n}$, 
\begin{itemize}
\item one may combine \eqref{eqn:aim11} and \eqref{eqn:CMO-a} to see when  $(2\kappa+1)\ell(Q)\leq 2^{-I_\varepsilon}$, we have 
\begin{equation}\label{eqn:aim12}
    \frac{1}{|Q|}\int_{Q} \left|   Mf(x)-(Mf)_Q  \right| dx\leq C \left[\varepsilon^{1/2}+ \|f\|_{{\rm BMO}} \,\varepsilon^{1/(2n)}  \right],
\end{equation}
so
$$
 \lim _{a \rightarrow 0} \sup _{Q: \,\ell_{Q} \leq a}\left(|Q|^{-1} \int_{Q}\left|Mf(x)-(Mf)_{Q}\right|  d x\right) =0.
$$
Applying a John--Nirenberg type inequality associated to small cubes with $(2\kappa+1)\ell(Q)\leq 2^{-I_\varepsilon}$, we conclude that 
$\widetilde{\gamma}_1(Mf)=0$. This approach can be used to verify $\widetilde{\gamma}_3(Mf)=\widetilde{\gamma}_5(Mf)=0$, that is, the $L^2$ integrals involved therein can be  replaced by the corresponding $L^1$ integrals.

\smallskip

\item Similarly, one may combine \eqref{eqn:aim11} and \eqref{eqn:CMO-b} to see when $Q^*= (2\kappa+1)Q\subseteq (Q(0, 2^{J_\varepsilon }))^c$, \eqref{eqn:aim12} still holds. Hence, $\widetilde{\gamma}_3(Mf)=0$.

\end{itemize}

\smallskip

It remains to show $\widetilde{\gamma}_5(Mf)=0$. Note that for any cube $Q$ with $\ell(Q)\geq \rho(c_Q)$,  there exists a constant $C >0$ independent of $Q$ such that there exists a sequence $\left\{Q(x_k,\rho(x_k))\right\}_{k}$ such that 
\begin{equation}\label{eqn:joint}
	 Q\subseteq \bigcup_k Q(x_k,\rho(x_k)) \quad \text{and}\quad \sum_k |Q(x_k,\rho(x_k))|\leq C|Q|.
\end{equation}
Hence, to verify $\widetilde{\gamma}_5(Mf)=0$ for any given $0\leq f\in {\rm CMO}_{\L}(\mathbb R^n)$, it suffices to show
\begin{equation}\label{eqn:gamma5}
   \lim _{a \rightarrow \infty} \sup _{\substack{Q: \, Q \subseteq (Q(0, a))^{c}\\   \ell(Q)= \rho(c_Q)}} \frac{1}{|Q|} \int_{Q}M f(x) d x =0.
\end{equation}

For any given $Q$ with $ \ell(Q)= \rho(c_Q)$ and for any $x\in Q$, define
$$
    M'_1f(x):=\sup_{Q' \ni  x:\, \ell(Q')<   \rho(c_Q)} f_{Q'}\quad 
    \text{and} \quad M'_2f(x):=\sup_{Q' \ni  x:\, \ell(Q')\geq   \rho(c_Q)} f_{Q'}.
$$
Clearly, $Mf=\max\left\{ M'_1 f, M'_2 f\right\}$ on $Q$, and so
$$
   \frac{1}{|Q|}\int_Q Mf(x) dx\leq \sum_{i=1}^2 \frac{1}{|Q|}\int_Q M'_i f(x)dx.
$$

Note that  
\begin{align}\label{eqn:M1-2}
   \frac{1}{|Q|}\int_Q M'_1f(x)dx &= \frac{1}{|Q|}\int_Q M'_1(f\mathsf 1_{3Q})(x)dx\nonumber\\
   &= \Bigg(\frac{1}{|Q|}\int_Q \big|M'(f\mathsf 1_{3Q})(x)\big|^2 dx\Bigg)^{1/2}\nonumber\\
   &\leq C\, \Bigg(\frac{1}{|Q|}\int_{\mathbb R^n} \big| f\mathsf 1_{3Q}(x)\big|^2 dx\Bigg)^{1/2} \nonumber\\
   &\leq C\, \Bigg(\frac{1}{|3Q|}\int_{3Q} \big| f(x)\big|^2 dx\Bigg)^{1/2}.
\end{align}

On the other hand, there exists a constant $C>0$ such that for any cube  $Q'\ni x$ with $\ell(Q')\geq   \rho(c_Q)$, $f_{Q'}\leq C f_{Q''}\leq CM'_2f(c_Q)$, where $Q''$ is the smallest cube containg $Q$ and $Q'$. Hence, 
$$
     \frac{1}{|Q|}\int_Q M'_2f(x)dx\leq C M'_2f(c_Q).
$$

For any given $\varepsilon>0$, let $J_{\varepsilon}$ be the positive integer as in \eqref{eqn:CMO-c}. For any $Q(y, \rho(y))\cap Q(0,2^{J_{\varepsilon}})\neq \emptyset$, it follows from Lemma \ref{lem:size-rho} that  $\rho(y)\approx \rho(z)$ for any $z\in Q(y, \rho(y))\cap Q(0,2^{J_{\varepsilon}})$, thus $Q(y, \rho(y))\subset Q(z, C \rho(z))$ for some   $C>1$ independent of $y$ and $z$. Without loss of generality, we may assume that 
\begin{equation}\label{eqn:rho-0}
	\rho(0)\leq 2^{J_{\varepsilon}},
\end{equation} 
which can be achieved by taking  $J_{\varepsilon}$ sufficiently large.
Then by Lemma \ref{lem:size-rho} again,
$$
    \sup_{x\in Q(0,2^{J_{\varepsilon}})} \rho(x)\leq c \rho(x_0)^{\frac{1}{k_0+1}} 2^{J_{\varepsilon}\cdot \frac{k_0}{k_0+1}}\leq c 2^{J_\varepsilon}
$$
for some   $c>1$. Hence,
\begin{equation}\label{eqn:inclusion}
	  \bigcup_{Q(y, \rho(y)):\, Q(y, \rho(y))\cap Q(0,2^{J_{\varepsilon}})\neq \emptyset}  Q(y,\rho(y))  \subseteq Q(0,C'2^{J_{\varepsilon}})
\end{equation}
for some $C'>1$.

Note that  $M_2'$ is bounded by its corresponding  centered maximal counterpart, thus when $c_Q$ is far away from the origin,
\begin{align*}
  M'_2f(c_Q) &\leq C\sup_{Q'=Q'(c_Q, \ell(Q')) :\, \ell(Q')\geq   \rho(c_Q)} f_{Q'}\\
  &=C\max\left \{ \sup_{Q'(c_Q, \ell(Q'))\subseteq( Q(0,2^{J_{\varepsilon}} ))^c   :\, \ell(Q')\geq   \rho(c_Q)} f_{Q'},  \   \sup_{Q'(c_Q, \ell(Q'))\cap Q(0,2^{J_{\varepsilon}} )   \neq \emptyset :\, \ell(Q')\geq   \rho(c_Q)} f_{Q'}  \right \}.
\end{align*}

For any $Q'(c_Q, \ell(Q'))$ with    $\ell(Q')\geq \rho(c_Q)  $,  if $Q'(c_Q, \ell(Q'))\subseteq( Q(0,2^{J_{\varepsilon}} ))^c$, it follows from \eqref{eqn:CMO-c} to see $f_{Q'}<\varepsilon$, as desired.

On the other hand,  $Q'(c_Q, \ell(Q'))\cap Q(0,2^{J_{\varepsilon}} )   \neq \emptyset$ and $\ell(Q')\geq \rho(c_Q)$. Similar to \eqref{eqn:joint},  there exists a sequence $\left\{Q(x_k,\rho(x_k))\right\}_{k}$ such that 
$$
	 Q'(c_Q, \ell(Q'))\subseteq \bigcup_k Q(x_k,\rho(x_k)) \quad \text{and}\quad \sum_k |Q(x_k,\rho(x_k))|\leq C|Q'(c_Q, \ell(Q'))|.
$$
Denote
$$
  \Lambda_1=\left\{k:\,  Q(x_k, \rho(x_k)) \subseteq( Q(0,2^{J_{\varepsilon}} ))^c \right\}  \quad   \text{and}\quad  \Lambda_2=\left\{k:\,  Q(x_k, \rho(x_k)) \cap  Q(0,2^{J_{\varepsilon}} )   \neq \emptyset \right\} .
$$
Combining  the definition \eqref{eqn:def-BMOL-norm}, \eqref{eqn:CMO-c} and \eqref{eqn:inclusion},   whenever $Q'(c_Q, \ell(Q'))\cap Q(0,2^{J_{\varepsilon}} )   \neq \emptyset$,
\begin{align*}
    f_{Q'(c_Q, \ell(Q'))} &\leq \frac{1}{|Q'(c_Q, \ell(Q'))|}   \sum_{k\in \Lambda_1} \int_{Q(x_k, \rho(x_k))}f(x)dx+ \frac{1}{|Q'(c_Q, \ell(Q'))|} \int_{\bigcup_{k\in \Lambda_2} Q(x_k, \rho(x_k) ) } f(x)dx   \\
      & \leq C \varepsilon \frac{\sum_{k\in \Lambda_1}  |Q(x_k, \rho(x_k))| }   {|Q'(c_Q, \ell(Q'))|} +   C \frac{|Q(0,C'2^{J_{\varepsilon}})| }{|Q'(c_Q, \ell(Q'))|}  f_{Q(0,C'2^{J_{\varepsilon}}) }\\
   & \leq C \varepsilon +   C \|f\|_{{\rm BMO}_{\L}}\frac{|Q(0,C'2^{J_{\varepsilon}}) | }{|Q'(c_Q, \ell(Q'))|},
\end{align*}
where the last inequality follows from  \eqref{eqn:rho-0} and $C'>1$.

To continue, observe that  there exists a positive $J'_{\varepsilon}>J_{\varepsilon}$ such that for any  $Q(c_Q,\rho(c_Q))\subseteq Q(0,2^{J'_{\varepsilon}})^c$, we have 
\begin{equation}\label{eqn:cond1}
    \frac{|Q(0,C'2^{J_{\varepsilon}})|}{|Q'(c_Q, \ell(Q'))|}<\varepsilon\quad \text{whenever}\ \     Q'(c_Q, \ell(Q'))\cap Q(0,2^{J_{\varepsilon}} )   \neq \emptyset,\,  \ell(Q')\geq \rho(c_Q).
\end{equation}

Indeed, since $|c_Q|\geq 2^{J'_{\varepsilon}}$, for any cube $Q'(c_Q, \ell(Q'))$ having nonnempty intersection with $Q(0,2^{J_{\varepsilon}} )$, we have 
 $$
    \ell(Q')\gtrsim |c_Q|-2^{J_{\varepsilon}}\geq  2^{J'_{\varepsilon}-1},
 $$
which ensures \eqref{eqn:cond1} by taking $J'_{\varepsilon}\gtrsim  J_{\varepsilon}-(\log_2 \varepsilon)/n$. Note that the condition ``$\ell(Q')\geq \rho(c_Q)$'' for such $Q'$ is also compatible: since $Q(c_Q,\rho(c_Q))\subseteq Q(0,2^{J'_{\varepsilon}})^c$,  there exists some integer $j\geq 0$ such that $|c_Q|\approx  2^{J'_{\varepsilon}+j}$, then by Lemma \ref{lem:size-rho} and \eqref{eqn:rho-0},  
$$
    \rho(c_Q)\lesssim 2^{(J'_{\varepsilon}+j)\cdot\frac{k_0}{k_0+1}},
$$
which implies that  
if $Q'(c_Q, \ell(Q'))\cap Q(0,2^{J_{\varepsilon}} )   \neq \emptyset$, we must have  $\ell(Q')\geq  2^{J'_{\varepsilon}+j-1}-2^{J_{\varepsilon}}\geq \rho(c_Q)$ by taking $J'_{\varepsilon}\gtrsim (k_0+1) J_{\varepsilon}$ sufficently large.

Therefore, for any $\varepsilon>0$, there exists positive integers  $J'_{\varepsilon}> J_{\varepsilon}$ such that for any cube $Q(c_Q,\rho(Q))\subseteq (Q(0,2^{J'_{\varepsilon}}))^c$, 
$$
    M'_2f(c_Q)\lesssim \varepsilon.
$$

From the above, \eqref{eqn:gamma5} holds, and  $\widetilde{\gamma}_5(Mf)=0$ follows readily.

We complete the proof of Theorem \ref{thm:M-CMO}.
\end{proof}

\smallskip

\begin{remark}
Let $R_j^*$ be the adjoint of the Riesz transform $R_j$ associated to $\L$ for $j=1,2\ldots, n$.
Recall that  (see \cite[(5.5)]{Shen1}) for $V\in {\rm RH}_q$ with $q> n/2$, 
$$
   \|R_j^* f\|_{L^p(\mathbb R^n)}\leq C_p\|f\|_{L^P(\mathbb R^n)} \quad \text{for}\    p'_0\leq p<\infty,
$$
where $p_0'=p_0/(p_0-1)$ and $1/(p_0)=(1/q)-(1/n)$.

For the endpoint case  $p=\infty$, we have 
$$
    \|R_j^* f\|_{{\rm BMO}_{\L}}\leq C \|f\|_{L^\infty(\mathbb R^n)}.
$$
More precisely, for each $g\in {\rm BMO}_{\L}(\mathbb R^n)$, we can represent $g$ as $\displaystyle g=\varphi_0+\sum_{j=0}^n R_j^* \varphi_j$, where $\varphi_j\in L^\infty(\mathbb R^n)$ for $0\leq j\leq n$, and $\displaystyle \|g\|_{{\rm BMO}_{\L}}\approx \sum_{j=0}^n \|\varphi_j\|_{L^\infty}$; see \cite[Theorem 1.3]{WY}.

Furthermore, one may combine \eqref{eqn:M-BMO} and  the John--Nirenberg type inequality to see for $1\leq p<\infty$ and  $f\in L^\infty(\mathbb R^n)$, 
$$
    \sup_{B=B(x_B,r_B):\, r_B\geq \rho(x_B)} \bigg(\frac{1}{|B|}\int_B |R_j^* f(x)|^p dx\bigg)^{1/p} \leq C_p \|f\|_{L^\infty}.
$$
\end{remark}

\begin{remark}\label{rem:CMO-M}
As a straightforward consequence,  when $V\equiv 0$ (i.e., $\rho(x)\equiv +\infty$), the above result implies that the Hardy--Littlewood maximal operator is bounded on the classical ${\rm CMO}(\mathbb R^n)$. Specifically, 
 for any $f$ belongs to ${\rm CMO}(\mathbb R^n)$ for which $Mf$ is not identically equal to infinity, then $Mf$ also belongs to ${\rm CMO}(\mathbb R^n)$, and 
$$
   \|Mf\|_{{\rm BMO}(\mathbb R^n)}\leq C \|f\|_{{\rm BMO}(\mathbb R^n)}.
$$
This result is a straightforward consequence of our argument by noting that  $f\in {\rm CMO}(\mathbb R^n)$ if and only if $f\in {\rm BMO}(\mathbb R^n)$ and $\widetilde{\gamma}_1(f)=\widetilde{\gamma}_2(f)=\widetilde{\gamma}_3(f)=0$, hence it suffices to prove $\widetilde{\gamma}_1(Mf)=\widetilde{\gamma}_2(Mf)=\widetilde{\gamma}_3(Mf)=0$. Among them, $\widetilde{\gamma}_1(Mf)=\widetilde{\gamma}_3(Mf)=0$ has been proved, and one may combine
\eqref{eqn:aim11}  and $\widetilde{\gamma}_2(f)=0$ to deduce \eqref{eqn:aim12} also holds for any cube whose sidelength is sufficiently large, that is, we obtain the remaining $\widetilde{\gamma}_2(Mf)=0$.

Now we address the fact that there exists $f\in {\rm CMO}(\mathbb R^n)$ such that $Mf\equiv +\infty$. 

We consider $n=1$ for simplicity. Define
$$
    f(x)=
    \begin{cases}
     \ln \ln |x|, \quad &  |x|\geq e,\\[6pt]
     0,\quad & |x|<e.
    \end{cases}
$$
It's clear $f$ is a uniformly continuous function on $\mathbb R$. As a consequence, $\widetilde{\gamma}_1(f)=0$.

We begin by verifying that  $f\in {\rm BMO}(\mathbb R)$.   To see it, we will use   
$$
   \|f\|_{{\rm BMO}(\mathbb R)}\approx \sup_{I:=[a,b]\subseteq \mathbb R} \inf_{  {\mathsf {Avg}}_I\in \mathbb R}  \frac{1}{b-a}\int_a^b \left|f(x)- {\mathsf {Avg}}_I\right|dx.
$$

Let $M$ be a sufficiently large integer. For  any interval $I=[a,b]$,\begin{itemize}
\item {\it Case I}.  $I\subseteq [-10M, 10M]$. Then 
$$
     \frac{1}{b-a} \int_a^b \left| f(x)-f_{[a,b]}\right| dx\leq 2\|f\|_{L^\infty([-10M, 10M])}.
$$

\item {\it Case II}.   $I\cap (\mathbb R\setminus  [-10M, 10M])\neq \emptyset$.

 Write $[a,b]=[c-R, c+R]$ with $c=\frac{a+b}{2}$ and $R=\frac{b-a}{2}$.

\begin{itemize}
\item Subcase 1. $I\cap [-M,M]=\emptyset$. 

Since $f$ is an even function, one may assume that $[a,b]\subseteq (M, +\infty)$. 
\smallskip
\begin{itemize}
\item 
{\it Subsubcase 1-1}.
If $c>2R$,   let $\mathsf{Avg}_I=\ln \ln c$. Combining $a=c-R>\max\{M, R\}$, we have 
\begin{align*}
&\frac{1}{b-a} \int_a^b \left| f(x)-\mathsf{Avg}_I\right| dx\\
\leq& \max\Big \{\ln \ln c- \ln \ln (c-R),  \ln \ln (c+R) -\ln \ln c  \Big \}\\
\leq &\max\left\{ \ln \left(1+  \frac{\ln(1+\frac{R}{a})}{\ln a}\right)   ,   \ln \left(1+  \frac{\ln(1+\frac{R}{c})}{\ln c}\right) \right\}\\
\leq& \ln \left(1+\frac{\ln 2}{\ln M}\right),
\end{align*}
which is sufficiently small.

\smallskip

\item 
{\it Subsubcase 1-2}.  Otherwise, if $c<2R$,  let $\mathsf{Avg}_I=\ln \ln R$.    
\begin{align*}
\frac{1}{b-a} \int_a^b \left| f(x)-\mathsf{Avg}_I\right| dx& = \frac{1}{2R}\int_{M<x<3R} \left|\ln \left( 1+ \frac{\ln \frac{x}{R}}{\ln R}\right)  \right|dx\\
&\lesssim \int_{0}^3     \left|\ln \left( 1+ \frac{\ln x}{\ln M}\right)  \right|dx\leq C_M,
\end{align*}
and $C_M$ is sufficiently small since $M>>1$.
\end{itemize}
Note that the argument in {\it Subcase 1} also implies that $\widetilde{\gamma}_3(f)=0$ (in fact we do not use the assumption ``$I\cap (\mathbb R\setminus  [-10M, 10M])\neq \emptyset$'').
\smallskip

\item Subcase 2.  $I\cap [-M,M]\neq\emptyset$. Then $R>9M/2$ and it's obvious that $|c|<2R$, thus $[a,b]\subseteq [-3R, 3R]$. Similar to {\it Subsubcase 1-2}, let $\mathsf{Avg}_I=\ln \ln R$ and    
\begin{align*}
\frac{1}{b-a} \int_a^b \left| f(x)-\mathsf{Avg}_I\right| dx& \leq \frac{1}{2R}\int_{e<|x|<3R} \left|\ln \left( 1+ \frac{\ln \frac{|x|}{R}}{\ln R}\right)  \right|dx\\
&\lesssim \int_{|x|<3}     \left|\ln \left( 1+ \frac{\ln |x|}{\ln M}\right)  \right|dx\leq 2C_M.
\end{align*}

Notably, we also complete the proof of   $\widetilde{\gamma}_2(f)=0$.
\end{itemize}
\end{itemize} 

From the above, we obtain $f\in {\rm BMO}(\mathbb R)$ and $\widetilde{\gamma}_1(f)=\widetilde{\gamma}_2(f)=\widetilde{\gamma}_3(f)=0$. Hence $f\in {\rm CMO}(\mathbb R)$, as desired.

\smallskip

For any $x\in \mathbb R$,  we may  assume  $x\geq 0$ since $f$ is even.  Consider the interval $I_R=[x, x+2R]$ for some $R>>1$,  
$$
  \frac{1}{2R}\int_{I_R} \left| f(y)\right|dy\geq \frac{1}{2R}\int_{x+R}^{x+2R} \ln \ln y \,dy\geq \frac{\ln \ln R}{2},
$$
which deduces that $Mf(x)=+\infty$ for every $x\in \mathbb R$. This is a remarkable phenomenon that   $f\in {\rm CMO}$ can not ensure that $Mf<+\infty$.  However, for any $f$ belongs to ${\rm CMO}$ for which $Mf$ is not identically equal to infinity, then $Mf$ is bounded on ${\rm CMO}$.
\end{remark}

\medskip

\section{Riesz transforms and ${\rm CMO}_{\L}(\mathbb{R}^n)$: proof of Theorem \ref{thm:Riesz-CMO}}
\label{sec:Riesz}
\setcounter{equation}{0}

 For each  $j=1,\ldots, n$, let $R_j(x,y)$ be  kernels of Riesz transforms $\displaystyle R_j=\frac{\partial}{\partial x_j}\L^{-1/2}$. Then the adjoint of $R_j$ is given by 
 $$
    R_j^* g(x)=\lim_{\varepsilon\to 0}\int_{|y-x|>\varepsilon}  R_j(y,x)g(y) dy.
 $$

\smallskip

\begin{proof}[Proof of Theorem~\ref{thm:Riesz-CMO}]
For each given $j=1,\ldots, n$, recall that $R_j^*$ is a bounded linear operator from $L^\infty(\mathbb{R}^n)$ to ${\rm BMO}_{\L}(\mathbb{R}^n)$. This, combined with  the fact that $C_0(\mathbb{R}^n)$ is the closure of $C_c^\infty(\mathbb{R}^n)$ in $L^\infty(\mathbb{R}^n)$, deduces that 
$$
      R_j^*(C_0(\mathbb{R}))\subseteq 
      \overline{R_j^* \left(C_c^\infty(\mathbb{R}^n)\right)}^{{\rm BMO}_{\L}},
$$
where $C_0(\mathbb{R}^n)$  is the space of all continuous
functions on $\mathbb{R}^n$ which vanish at infinity.
Meanwhile, it follows from Theorem~C in \cite{SW} that the spaces ${\rm CMO}_{\L}(\mathbb{R}^n)$ is  the closure in the ${\rm BMO}_{\L}(\mathbb{R}^n)$ norm of $C_0(\mathbb{R}^n)$ (i.e., (iii) of Theorem \ref{thm:CMO-dual-known}). 
Therefore, to prove Theorem~\ref{thm:Riesz-CMO}, it suffices to show
\begin{equation}\label{eqn:Riesz-Compact}
    R_j^* \left(C_c^\infty(\mathbb{R}^n)\right)\subseteq C_0(\mathbb{R}^n).
\end{equation}

\smallskip

\noindent{\it Step I.}\  \
Let  $R_j^0(x,y)$ be  kernels of the classical Riesz transforms $\displaystyle R_j^0 =\frac{\partial}{\partial x_j}(-\Delta)^{-1/2}$, $j=1,\ldots, n$.
Since $V\in RH_q$ for some $q\geq n/2$, we may assume that $n/2 \leq  q<n$  (the case for $q\geq n$ is simpler because the relevant kernels then exhibit better regularity). Notably, one remarkable fact is the self-improvement of the $RH_q$ class  that  for any $V\in RH_q$,  there exists $\varepsilon>0$ depending only on $C$ in \eqref{eqn:Reverse-Holder} and the dimension $n$ such that $V\in RH_{q+\varepsilon}$. Therefore,  $V\in B_{q_1}$ for some  $n/2\leq q<q_1<n$.

Applying Lemmas 5.7 and 5.8 in \cite{Shen1}, we have  for each $\varphi\in C_c^\infty (\mathbb{R}^n)$,
\begin{align}\label{eqn:aux-Riesz-1}
	\left|R_j^*(\varphi)(x)\right|&\leq \left| \int_{|y-x|> \rho(x)} R_j(y,x)\,\varphi(y)\, dy\right| +\left|  \int_{ |y-x|\leq \rho(x_0)} \left[R_j (y,x)-R_j^0 (y,x)\right]\varphi (y)\, dy\right| \nonumber\\
	&\quad + \left|\int_{|y-x|\leq \rho(x)}  R_j^0(y,x)\, \varphi(y) \, dy \right|\nonumber\\
	&\leq  C \left[M\left(|\varphi|^{p_1 '}\right)(x)\right]^{1/{p_1 '}} +2 \sup_{\tau>0}  \left| \int_{|y-x|>\tau}  R_j^0(y,x) \,\varphi(y)   \,dy   \right|,
\end{align}
where $\displaystyle \frac{1}{p'_1}=1-\frac{1}{q_1}+\frac{1}{n}$, and $M$ denotes the uncentered Hardy--Littlewood maximal operator. It's clear that
$$
\left\|   \left[M\left(|\varphi|^{p_1 '}\right)\right]^{1/{p_1 '}}  \right\|_{ L^\infty }\leq \|\varphi\|_{L^\infty}\   \   \   {\rm and }     \   \
\left[M\left(|\varphi|^{p_1 '}\right)(x)\right]^{1/{p_1 '}} \to 0\   \      {\rm as} \   \    |x|\to \infty,
$$
provided by $\varphi\in C_c^\infty(\mathbb{R}^n)$.
Moreover,  by Cotlar's inequality,
\begin{equation}\label{eqn:maximal-Riesz-1}
	\sup_{\tau>0}  \left| \int_{|y-x|>\tau}  R_j^0(y,x)\, \varphi(y)   dy   \right|\leq  C \left[  M\left(R_j^0(\varphi)\right)(x) +M(\varphi)(x)\right].
\end{equation}
It's clear that  $R_j^0(\varphi)\in C_0(\mathbb{R}^n)$ since the Fourier transform of $R_j^0(\varphi)$ belongs to $L^1(\mathbb{R}^n)$. This allows us to verify readily that the left hand side of \eqref{eqn:maximal-Riesz-1} is a $L^\infty$ function and vanishes at infinity, as desired.

 From the above, $R_j^*(\varphi)\in L^\infty$ and $\displaystyle \lim_{|x|\to \infty} R_j^*(\varphi)=0$ for any given $\varphi\in C_c^\infty(\mathbb{R}^n)$.

\smallskip

\noindent{\it Step II.}\  \
It remains to show $R_j^*(\varphi)$ is continuous. To this end, it suffices to show for any given $x_0\in \mathbb{R}^n$ and $\varepsilon>0$,  there exists a positive constant $\theta=\theta(x_0,\varepsilon)$ sufficiently small, such that for any $x_1\in B(x_0,\theta)$,
\begin{equation}\label{eqn:continue-Riesz-goal}
    \left|R_j^*(\varphi)(x_1)-R_j^* (\varphi)(x_0)\right|\lesssim \varepsilon  \quad {\rm for}\    \   x_1\in B(x_0,\theta).
\end{equation}

Firstly, we fix a positive integer $\kappa_0>>1$ such that
\begin{equation}\label{eqn:assume1}
      2^{-\kappa_0 (2-{n/q_1})} <\frac{\varepsilon}{   \Big\|\left[M\left(|\varphi|^{p_1 '}\right)\right]^{1/{p_1 '}} \Big\|_{L^\infty}}.
\end{equation}
For every $j=1,\ldots,n$ and $x\in B(x_0,2^{-\kappa_0}\rho(x_0))$, rewrite
\begin{align*}
	R_j^*(\varphi)(x)&= \int_{|y-x|> 2^{-\kappa_0}\rho(x_0)} R_j(y,x)\,\varphi(y)\, dy + \lim_{\tau\to 0} \int_{\tau <|y-x|\leq 2^{-\kappa_0}\rho(x_0)} \left[R_j (y,x)-R_j^0 (y,x)\right]\varphi (y)\, dy \\
	&\quad + \left[ ( R_j^0)^*(\varphi)(x) - \int_{|y-x|> 2^{-\kappa_0}\rho(x_0)}  R_j^0(y,x)\, \varphi(y) \, dy\right]\\
	&=: T_1(\varphi)(x)+T_2(\varphi)(x)+T_3(\varphi)(x),
\end{align*}
where $( R_j^0)^*(\varphi)(x)=- R_j^0 (\varphi)(x)$ due to the anti-symmetric property of the kernel of $R_j^0$.

We observe the following facts:

\begin{itemize}
	\item  Note that   (see \cite[p. 540]{Shen1})
$$
    \left(\int_{2^{j-1} \rho(x)<|y-x|\leq 2^{j}\rho(x)  }   \left|R_j (y,x)-R_j^0 (y,x)\right|^{p_1} dy   \right)^{1/{p_1}}\leq C (2^j)^{2-(n/{q_1})}(2^j \rho(x))^{-n/{p'_1}}\quad \text{for}\   j\leq 0,
$$	
and   $\rho(x)\approx \rho(x_0)$ for $x\in B(x_0, 2^{-\kappa_0}\rho(x_0))$, thus there exists a positive integer $\mathsf M$ such that 
\begin{align*}
   & \left(\int_{2^{j-1} \rho(x_0)<|y-x|\leq 2^{j}\rho(x_0)  }   \left|R_j (y,x)-R_j^0 (y,x)\right|^{p_1} dy   \right)^{1/{p_1}}\\
   \leq &  \left(\int_{2^{j-\mathsf M } \rho(x)<|y-x|\leq 2^{j+\mathsf M}\rho(x)  }   \left|R_j (y,x)-R_j^0 (y,x)\right|^{p_1} dy   \right)^{1/{p_1}}\\
   \leq& C (2^j)^{2-(n/{q_1})}(2^j \rho(x_0))^{-n/{p'_1}}\quad \text{for}\   j\leq 0.
\end{align*}
Consequently, in combination with \eqref{eqn:assume1},
\begin{align*}
	   &  |T_2(\varphi)(x)|\\
	   \leq &\int_{|y-x|\leq C 2^{-\kappa_0}\rho(x)} \left|R_j (y,x)-R_j^0 (y,x)\right| \, |\varphi (y)|\, dy\\
	     \leq& \sum_{j=-\infty}^{-\kappa_0} \left( \int_{|y-x|\leq 2^{j}\rho(x_0) } |f(y)|^{p'_1} dy\right)^{1/{p_1}}   \left(\int_{2^{j-1} \rho(x_0)<|y-x|\leq 2^{j}\rho(x_0)  }   \left|R_j (y,x)-R_j^0 (y,x)\right|^{p_1} dy   \right)^{1/{p_1}} \\
	     \leq& C 2^{-\kappa_0 (2- {n/q_1})} \left[M\left(|\varphi|^{p_1 '}\right)(x)\right]^{1/{p_1 '}}\\
	     \leq &  C\varepsilon
\end{align*}
	for any $x\in B(x_0, 2^{-\kappa_0}\rho(x_0))$.
	
	\smallskip
	
	\item For  the continuity of   $T_3(\varphi)$,   note that $(R_j^0)^*(\varphi)=-R_j^0(\varphi)\in C_0(\mathbb{R}^n)$ mentioned in {\it Step I}, it remains to show the continuity of 
$$
T_4(\varphi):=\int_{|y-x|> 2^{-\kappa_0}\rho(x_0)}  R_j^0(y,x) \,\varphi(y) \, dy,
$$
which is a bounded function by combining \eqref{eqn:maximal-Riesz-1} and the fact $R_j^0(C_c^\infty(\mathbb{R}^n))\subseteq C_0(\mathbb{R}^n)$.
	
	\smallskip
	
	Denote $E_x:=
   \{y:\, |y-x|\geq 2^{-\kappa_0}\rho(x)  \}$ for  $x\in B(x_0,2^{-\kappa_0}\rho(x_0))$, then
   $$
      \Delta_{\mu}  : \   \    \  x\mapsto  \left|\Delta E(x)\right|:= \left | \left(E_x\cap E_{x_0}\right )\setminus\left (E_x\cap E_{x_0}\right) \right|
   $$
is continuous on $B(x_0, 2^{-\kappa_0}\rho(x_0))$ and $\Delta_{\mu}(x_0)=0$. This, combined with   the explicit expression of $R_j^0(y,x)$, deduces that $T_4(\varphi)$ is continuous on $x_0$, whenever $\varphi\in C_c^\infty(\mathbb{R}^n)$.

\end{itemize}

\smallskip

From the above, \eqref{eqn:continue-Riesz-goal} follows readily if one could prove  that for any given $x_0\in \mathbb{R}^n$, there exists $\theta<2^{-\kappa_0 }\rho(x_0)$ such that
\begin{equation}\label{eqn:continue-Riesz-goal-step}
    \big|T_1(\varphi)(x_1)-T_1 (\varphi)(x_0)\big|\lesssim \varepsilon  \quad {\rm for}\    \   x_1\in B(x_0,\theta).
\end{equation}

\smallskip

\noindent{\it Step III.}\  \
Now let's verify the continuity of $T_1(\varphi)$ by proving \eqref{eqn:continue-Riesz-goal-step}.

Given $x_0\in \mathbb{R}^n$,  for each $x_1\in B(x_0,2^{-\kappa_0}\rho(x_0))$,
$$
     T_1(\varphi)(x_1)-T_1(\varphi)(x_0)=\int_{|y-x_0|> 2^{-\kappa_0}\rho(x_0)}  \left[  R_j(y,x_1)-R_j(y,x_0)\right]   \varphi(y) \,dy +\mathcal{E}(x_1),
$$
where
$$
    \big|\mathcal{E}(x_1)\big|\leq \int_{\Delta E(x_1)} \left|R_j(y,x_1)\right|\left|\varphi(y)\right |\, dy \to 0  \quad {\rm as}\  \    \  x_1\to x_0.
$$
Hence it remains to show
$$ 
    \int_{|y-x_0|> 2^{-\kappa_0}\rho(x_0)}  \left[  R_j(y,x_1)-R_j(y,x_0)\right]   \varphi(y) \,dy   \to 0 \quad {\rm as}\  \    \  x_1\to x_0.
$$ 
Let $\Gamma(x,y,\tau)$ denote the fundamental solution for the Schr\"odinger operator $-\Delta+(V+i\tau)$, $\tau\in \mathbb{R}$. Clearly,
$$
    \Gamma(x,y,\tau)=\Gamma(y,x,-\tau),
$$
and $\nabla_y \Gamma(x,y,\tau)$ is a solution to the equation $-\Delta u+(V+i\tau)u=0$ in $\mathbb{R}^n  \setminus  \{ y\}$ as a function of $x$.  Consequently, $\nabla_y \Gamma(y,x,\tau)$ is a solution to the equation $-\Delta u+(V+i(-\tau))u=0$ in $B(x_0, 2^{-(\kappa_0+1)}\rho(x_0))$, whenever $|y-x_0|> 2^{-\kappa_0}\rho(x_0)$.  Denote
$$
\delta= 2-{n/{q_1}}>0\quad \text{and}\quad r_0=2^{-(\kappa_0+2)}\rho(x_0).
$$ 
By the imbedding theorem of Morrey and \cite[Lemma 4.6]{Shen1} (see also the last inequality in \cite[page 534]{Shen1}), we have 
for any $x_1\in B(x_0,2^{-(\kappa_0+4)}\rho(x_0))=B(x_0,r_0/4)$,  
\begin{align*}
  & \left|\nabla_y \Gamma(y, x_1,\tau)-\nabla_y \Gamma(y,x_0,\tau)\right|\\
  \leq & C |x_1-x_0|^{1-n/{p_1}} \left(\int_{B(x_0,r_0)} \left| \nabla_x \nabla_y \Gamma(y, x,\tau) \right|^{p_1}dx\right)^{1/{p_1}}\\
  \leq& C \left(\frac{|x_1-x_0|}{r_0}\right)^{\delta }  \left\|\nabla_y \Gamma(y,x,\tau)\right\|_{L_x^\infty  (B(x_0,2r_0) )}   \left[1+\frac{1}{r_0^{n-2}}\int_{B(x_0,2r_0)} V(x) dx  \right],
\end{align*}
where $p_1$ is the index in  \eqref{eqn:aux-Riesz-1}. Alternatively,   one can apply  a similar argument to that of \cite[Proposition~2.12]{JL} to show the H\"older continuity.
 
To continue, we address the following facts.
\begin{itemize}
	\item By Lemma~1.2 in \cite{Shen1},
$$
   \frac{1}{r_0^{n-2}}\int_{B(x_0,2r_0)} V dx\leq C 2^{-\kappa_0 \delta}\lesssim 1.
$$

\item It's  known that $\rho(x)\approx \rho(x_0)$ whenever $|x-x_0|\lesssim \rho(x_0)$. Now regard $  \Gamma(y,x,\tau)$ is a solution to the equation $-\Delta u+(V+i\tau)u=0$ in $\mathbb{R}^n  \setminus  \{ x\}$ as a function of $y$, and $B(y, |y-x_0|/4)\cap B(x_0,2r_0)=\emptyset$.
Hence it follows from combining (4.8) and Theorem~2.7 in \cite{Shen1} that for each $m\in \mathbb{N}$, there exists a constant $C_m>0$ such that
\begin{align*}
    &  \left\|\nabla_y \Gamma(y,x,\tau)\right\|_{L_x^\infty  (B(x_0,2r_0) )} \\
    &\leq \frac{C_m}{\left(1+|\tau|^{1/2}R_y\right)^m   \left(1+R_y / \rho(x_0)\right)^m} \left[\frac{1}{R_y ^{n-2}}\int_{B(y,R_y)}  \frac{V(z)}{|z-y|^{n-1}}dz  +\frac{1}{R_y ^{n-1}}\right],
\end{align*}
where $R_y:=|y-x_0|/4$. 

\end{itemize}

Therefore, when $x_1\in B(x_0, r_0/4)$ and $|y-x_0|> 2^{-\kappa_0}\rho(x_0)=4r_0$, one may combine these estimates above  to obtain
\begin{align*}
    & \left|\nabla_y \Gamma(y, x_1,\tau)-\nabla_y \Gamma(y,x_0,\tau)\right| \\
    \leq \, &   \frac{C_m}{\left(1+|\tau|^{1/2}R_y\right)^m  \left(1+R_y / \rho(x_0)\right)^m}   \left(\frac{|x_1-x_0|}{r_0}\right)^\delta      \left[\frac{1}{R_y ^{n-2}}\int_{B(y,R_y)}  \frac{V(z)}{|z-y|^{n-1}}dz  +\frac{1}{R_y ^{n-1}}\right].
 \end{align*}
Furthermore,
\begin{align*} 
	\left|R_j(y,x_1)-R_j(y,x_0)\right|&= \left|  -\frac{1}{2\pi}\int_{\mathbb{R}} (-i\tau)^{-1/2} \left[\nabla_y \Gamma(y,x_1,\tau) -\nabla_y\Gamma(y,x_0,\tau) \right]  d\tau\right|\nonumber\\
	&\leq \frac{C_m}{\left(1+R_y/ \rho(x_0)\right)^m }  \left(\frac{|x_1-x_0|}{r_0}\right)^\delta    \left[\frac{1}{R_y ^{n-1}}\int_{B(y,R_y)}  \frac{V(z)}{|z-y|^{n-1}}dz  +\frac{1}{R_y ^ n}\right].
\end{align*} 

Note that
$$
      \Big\{y\in \mathbb{R}^n:\    |y-x_0|>2^{-\kappa_0}\rho(x_0) \Big\} =\bigcup_{k=-\kappa_0+1}^\infty \Big\{y\in \mathbb{R}^n:\     2^{k-1}\rho(x_0)< |y-x_0|\leq 2^{k}\rho(x_0) \Big\},
$$
and for each $k\geq \kappa_0+1$, it follows from  the Hardy--Littlewood--Sobolev inequality to obtain
\begin{align*}
  &   \left|\int_{ 2^{k-1}\rho(x_0)<|y-x_0|\leq 2^k \rho(x_0)  }  \left[  R_j(y,x_1)-R_j(y,x_0)\right]   \varphi(y) \,dy\right|\\
  \leq \, &  \frac{C_m} {\left(1+2^k\right)^m }  \left(\frac{|x_1-x_0|}{r_0}\right)^\delta   \left\{ \frac{1}{(2^k \rho(x_0)) ^{n-1}} \int_{2^{k-1}\rho(x_0)<|y-x_0|\leq 2^k \rho(x_0)}    \left(\int_{\mathbb R^n}  \frac{V(z)\mathsf 1_{B(y,R_y)}(z)}{|z-y|^{n-1}}dz \right)  |\varphi(y)|dy  +M\varphi(x_0)\right\}
    \\ 
     \leq \, &  \frac{C_m} {\left(1+2^k\right)^m }  \left(\frac{|x_1-x_0|}{r_0}\right)^\delta   \left\{ \frac{1}{(2^k \rho(x_0)) ^{n-1}}  \bigg(\int_{2^{k-1}\rho(x_0)<|y-x_0|\leq 2^k \rho(x_0)}    \left|\int_{B(y,R_y)}  \frac{V(z)} {|z-y|^{n-1}}dz \right|^{p_1 }dy\bigg)^{1/{p_1}} \right.\\
    &\hspace{6cm} \left.\cdot \left(\int_{|y-x_0|\leq 2^k \rho(x_0)} |\varphi(y)|^{p_1 '}dy \right)^{1/{p_1 '}}   +M\varphi(x_0)\right\}\\
    \leq \,&    \frac{C_m} {\left(1+2^k\right)^m }  \left(\frac{|x_1-x_0|}{r_0}\right)^\delta    \frac{1}{(2^k \rho(x_0)) ^{n-1}}   \left(\int_{2^{k-2}\rho(x_0)<|y-x_0|\leq 2^{k+1}\rho(x_0)} V(z)^{q_1} dz\right)^{1/{q_1}}  \left[ M\left(|\varphi|^{p_1 '}\right)(x_0)\right]^{1/{p_1 '}} \left(2^k \rho(x_0) \right)^{n/{p_1 '}}\\
     &\  + \frac{C_m} {\left(1+2^k\right)^m }  \left(\frac{|x_1-x_0|}{r_0}\right)^\delta   M\varphi(x_0).
\end{align*}

Moreover, the reverse H\"older inequality possessed by $V\in RH_{q_1}$ 
deduces
$$
   \left(\int_{2^{k-2}\rho(x_0)<|y-x_0|\leq 2^{k+1}\rho(x_0)} V(z)^{q_1} dz\right)^{1/{q_1}}  \leq  C \left(2^{k}\rho(x_0)\right)^{n/{q_1}-2} \frac{1}{    \left(2^{k}\rho(x_0)\right)^{n-2}}\int_{B(x_0,2^{k+1}\rho(x_0))} V dy.
$$
Moreover, by using  the doubling property (1.1)   in \cite{Shen1},  we have
\begin{align}\label{eqn:aux-potential-Lp}
   \left(\int_{2^{k-1}\rho(x_0)<|y-x_0|\leq 2^{k}\rho(x_0)} V(y)^{q_1} dy\right)^{1/{q_1}}
   & \leq
   \begin{cases}\displaystyle
   \left (2^k \rho(x_0)\right)^{n/{q_1}-2} ,   &  \mbox{if }  \  k< 0;\\[10pt]
   	\left (2^k \rho(x_0)\right)^{n/{q_1}-2}C_0^k,   &  \mbox{if }  \  k\geq 0,
   \end{cases}
    \end{align}
   where $C_0>1$ is the doubling constant in (1.1) in \cite{Shen1}.
   Without loss of generality, assume that $C_0>2$ and take $m=2\cdot \log_2 C_0$ such that for any $k\geq 0$,
   $$
           \frac{C_m}{\left(1+2^k\right)^m } C_0^k\leq \frac{C} {2^{k\log_2 C_0}},
   $$
  where $C$ is a positive constant independent of $k\geq 0$.

Therefore, for any given $\varphi\in C_c^\infty(\mathbb{R}^n)$ and  $x_0\in \mathbb{R}^n$,
\begin{align*}
  &   \left|\int_{|y-x_0|> 2^{-\kappa_0}\rho(x_0)}  \left[  R_j(y,x_1)-R_j(y,x_0)\right]   \varphi(y) \,dy\right|\\
  \leq\,  &  \sum_{k=-\kappa_0 +1}^\infty \frac{C}{\max\{1,2^{k\log_2 C_0}\} }  \left(\frac{|x_1-x_0|}{2^{-(\kappa_0+1)}\rho(x_0)}\right)^\delta  \left[ M\left(|\varphi|^{p_1 '}\right)(x_0)\right]^{1/{p_1 '}}  \\
  &\quad +   \sum_{k=-\kappa_0 +1}^\infty  \frac{C} {1+2^k } \left(\frac{|x_1-x_0|}{2^{-(\kappa_0+1)}\rho(x_0)}\right)^\delta    M\varphi(x_0)\\
  \leq \, &  C \frac{\kappa_0}{\left(2^{-(\kappa_0+1)}\rho(x_0)\right)^\delta}   \left\{   \left[ M\left(|\varphi|^{p_1 '}\right)(x_0)\right]^{1/{p_1 '}}  + M\varphi(x_0) \right\}  |x_1-x_0|^\delta     \to 0 \quad {\rm as}\  \    \  x_1\to x_0.
\end{align*}
That is,  \eqref{eqn:continue-Riesz-goal-step} holds and  $T_1(\varphi)$ is continuous on arbitrary given $x_0$.

This, combined with the result in {\it Step I}, deduces that  $T_1(\varphi)\in C_0(\mathbb{R}^n)$, as desired.

 \smallskip

 Therefore, we complete the proof of Theorem~\ref{thm:Riesz-CMO}.
\end{proof}

\smallskip

As a consequence, we have  the following representation based on Theorem~\ref{thm:Riesz-CMO}.

\begin{lemma}\label{lem:Riesz-CMO-2}
Suppose $V\in RH_q$ for some $q\geq n/2$.
For every continuous linear functional $\ell$ on the ${\rm CMO}_{\L}(\mathbb{R}^n)$ space, there exists a uniquely finite Borel measure $\mu_0$ whose Riesz transforms $R_j(d \mu_0)(x)=\int R_j (x,y)\,d\mu_0(y)$ associated to $\L$ for $j=1,2,\ldots, n$ are all finite Borel measures, 
 such that the functional $\ell$ can be realized by 
$$
     \ell(g)=\int_{\mathbb{R}^n}   g(x)\, d\mu_0(x),
$$
which is initially defined on  the dense subspace $C_0(\mathbb{R}^n)$, and has a unique extension to ${\rm CMO}_{\L}(\mathbb{R}^n)$.
\end{lemma}

\begin{proof}
Given  $\ell\in ({\rm CMO}_{\L})^*$, then there exists a constant $c>0$ such that
$$
    |\ell(g)|\leq c\|g\|_{{\rm BMO}_{\L}} \quad    {\rm for \     \   \     all     } \  \    \     g\in {\rm CMO}_{\L}.
$$
Notice that ${\rm CMO}_{\L}(\mathbb R^n)$ is the closure of $C_0(\mathbb{R}^n)$ in the ${\rm BMO}_{\L}(\mathbb R^n)$ norm and the space $C_0(\mathbb{R}^n)$ is equipped with the supremum norm, clearly for each $\ell\in ({\rm CMO}_{\L}(\mathbb R^n))^*$,
$$
      |\ell(g)|\leq c\|g\|_{{\rm BMO}_{\L}}\leq 2c \|g\|_{L^\infty} \quad    {\rm for \     \   \     all     } \  \    \     g\in C_0(\mathbb{R}^n).
$$
That is,  $\ell$ is also a bounded linear functional on $C_0(\mathbb{R}^n)$. Hence it follows from the Riesz representation theorem (see \cite[Section~6.19]{Ru} for instance)  that
 there exists a uniquely regular (complex-valued)  Borel measure $\mu_0$ whose total variation $|\mu_0|(\mathbb{R}^n)<\infty$, such that
\begin{equation}\label{eqn:dual-CMO-aux-1}
    \ell(g)=\int_{\mathbb{R}^n}  g(x)\,d\mu_0(x)=:\mu_0(g)  \quad    {\rm for \     \   \     all     } \  \    \     g\in C_0(\mathbb{R}^n).
\end{equation}
In turn, since $C_0(\mathbb{R}^n)$ is dense in ${\rm CMO}_{\L}(\mathbb R^n)$,  the linear functional $\mu_0$ given by \eqref{eqn:dual-CMO-aux-1} initially defined on $C_0(\mathbb{R}^n)$ has a unique extension to ${\rm CMO}_{\L}(\mathbb R^n)$. Thus every  $\ell\in ({\rm CMO}_{\L}(\mathbb R^n))^*$ can be realized by  a uniquely finite Borel measure $\mu_0$. In the sequel we fix such $\ell$ and $\mu_0$.

Moreover, It follows from combining \eqref{eqn:dual-CMO-aux-1} and $\ell\in ({\rm CMO}_{\L}(\mathbb R^n))^*$ that
\begin{equation}\label{eqn:dual-CMO-estimate}
   |\mu_0(g)| =|\ell(g)|\leq c \|g\|_{{\rm BMO}_{\L}}  \quad    {\rm for \     \   \     all     } \  \    \     g\in C_0(\mathbb{R}^n).
\end{equation}
We aim to characterize properties of the measure  $\mu_0$ from the perspective of Riesz transforms, motivated by the analogous result for the Laplacian operator in place of $\L$.

To this end, note  that  the linear operator $R_j^*:  \, C_0 (\mathbb{R}^n) \to {\rm CMO}_{\L}(\mathbb R^n)$ is bounded by Theorem~\ref{thm:Riesz-CMO}, and $C_0(\mathbb{R}^n)$ and ${\rm CMO}_{\L}(\mathbb R^n)$ are both Banach spaces, 
 so  the operator $R_j$,  as the adjoint of $R_j^*$, satisfies
$$
    R_j\left(  \left({\rm CMO}_{\L}\right)^*\right)  \subseteq (C_0)^*.
$$
Alternatively, the above inclusion can be deduced by recalling that $R_j:\, H_{\L}^1 = \left({\rm CMO}_{\L}\right)^* \to L^1$ is bounded (see \cite{DZ} for instance).
Hence $R_j(\ell)$ is a  bounded linear functional on  $C_0$ by means of 
\begin{equation}\label{eqn:dual-CMO-Riesz-aux-2}
    \langle R_j(\ell),\, g\rangle=\langle \ell, R_j^*(g)\rangle
     =\ell(R_j^*(g))   \quad    {\rm for \     \   \     all     } \  \    \     g\in C_0(\mathbb{R}^n).
\end{equation}
This, combined with  $R_j^*(C_c^\infty)\subseteq C_0$ (i.e., \eqref{eqn:Riesz-Compact} in the proof of  Theorem~\ref{thm:Riesz-CMO}) and the representation \eqref{eqn:dual-CMO-aux-1}, implies that
$$
   \ell(R_j^*(\phi)) = \mu_0(R_j^*(\phi)) =\langle R_j(\mu_0),\, \phi\rangle
      \quad    {\rm for \     \   \     all     } \  \    \     \phi\in C_c^\infty(\mathbb{R}^n). 
$$
That is,
\begin{equation}\label{eqn:dual-CMO-Riesz-aux-3}
  \int_{\mathbb{R}^n} R_j^* (\phi)(x)\, d\mu_0(x) 
    =  \int_{\mathbb{R}^n} \left[\int_{\mathbb{R}^n} R_j(x,y)\, d\mu_0(y)\right]  \phi(x)   \, dx \quad    {\rm for \     \   \     all     } \  \    \     \phi\in C_c^\infty(\mathbb{R}^n). 
\end{equation}

On the other hand, since $R_j(\ell)$ is  a bounded linear functional on $C_0(\mathbb{R}^n)$, by the Riesz representation theorem again,  there exists a finite Borel measure $\mu_j$ such that
$$
    \langle R_j(\ell)  ,g\rangle=\int_{\mathbb{R}^n}  g(x)\, d\mu_j(x)  \quad    {\rm for \     \   \     all     } \  \    \     g\in C_0(\mathbb{R}^n).
$$
In particular, using $R_j^*(C_c^\infty)\subseteq C_0$ again,
\begin{align*}
	   \int_{\mathbb{R}^n}\phi(x)\, d\mu_j(x)    &= \langle R_j(\ell)  ,\phi\rangle= \ell(R_j^*(\phi))=\int_{\mathbb{R}^n} R_j^* (\phi)(x)\, d\mu_0(x)  \quad    {\rm for \     \   \     all     } \  \    \     \phi\in C_c^\infty(\mathbb{R}^n). 
\end{align*}
This, together with \eqref{eqn:dual-CMO-Riesz-aux-3}, deduces that
$$
    \int_{\mathbb{R}^n} \phi(x) \left[\int_{\mathbb{R}^n} R_j(x,y)\, d\mu_0(y)\right]dx =\int_{\mathbb{R}^n}\phi(x)\, d\mu_j(x)   \quad    {\rm for \     \   \     all     } \  \    \     \phi\in C_c^\infty(\mathbb{R}^n). 
$$

Then a standard argument by contradiction shows that 
\begin{equation}\label{eqn:dual-CMO-Riesz-aux-4}
   R_j(d\mu_0)(x) \, dx=d\mu_j(x),\   \   {\rm i.e.,} \  \   \mu_j=R_j(d\mu_0).
\end{equation}

Therefore, for any given $\ell\in ({\rm CMO}_{\L})^*$, it can be realized by a   uniquely finite Borel measure $\mu_0$, whose Riesz transforms $R_j(\mu_0)$ for $j=1,2,\ldots, n$ are all finite Borel measures.  The proof is end.
\end{proof}

\medskip

\begin{remark}  \textbf{(Riesz transforms and subharmonicity)}\\
(i).
When $V\equiv 0$,  then ${\rm CMO}_{\L}(\mathbb R^n)={\rm CMO}_{-\Delta}(\mathbb R^n)={\rm CMO}(\mathbb R^n)$ is known by (ii) of Theorem~\ref{thm:CMO-dual-known} (since ${\rm CMO}(\mathbb R^n)$ is the closure of $C_c^\infty(\mathbb R^n)$ in the ${\rm BMO}(\mathbb R^n)$ norm).  In this case, Lemma~\ref{lem:Riesz-CMO-2} says that every continuous linear functional on ${\rm CMO}$ can be realized by a finite measure $\mu_0$ whose classical Riesz transforms $R_j^0(d\mu_0)$ for $j=1,\ldots, n$ are all finite measures. 

Hence, it follows from the F. and M. Riesz theorem (see Corollary 1 in \cite[p. 221]{St1} for instance) that  there exists  a function $f\in H^1(\mathbb{R}^n)$ such that $d\mu_0(x)=f(x)dx$, where $H^1(\mathbb{R}^n)$  is the classical Hardy space.

That is, we obtain 
$
     ({\rm CMO}(\mathbb{R}^n))^*\subseteq H^1(\mathbb{R}^n),
$ 
as a straightforward consequence of Lemma~\ref{lem:Riesz-CMO-2} by taking $V\equiv 0$. Note that the reverse inclusion $H^1(\mathbb{R}^n) \subseteq  ({\rm CMO}(\mathbb{R}^n))^*$ is trivial by combining $(H^1(\mathbb{R}^n) )^*={\rm BMO}(\mathbb{R}^n)$ and ${\rm CMO}(\mathbb{R}^n)\subsetneqq {\rm BMO}(\mathbb{R}^n)$. Hence Lemma~\ref{lem:Riesz-CMO-2} implies the classical well-known result (see  \cite[Proposition 3.5]{Ch}  for instance)
$$
     ({\rm CMO}(\mathbb{R}^n))^*= H^1(\mathbb{R}^n).
$$

Notably, we remind  that a crucial ingredient to show  the F. and M. Riesz theorem is the subharmonicity of $|F|^p$ for $p\geq (n-1)/n$, where 
$$
     F(x,t)=\left(e^{-t\sqrt{-\Delta}} (d\mu_0)(x),e^{-t\sqrt{-\Delta}} \left(R_1^0(d\mu_0)\right)(x),\ldots, e^{-t\sqrt{-\Delta}} \left(R_n^0(d\mu_0)\right)(x)  \right), \   \  (x,t)\in \mathbb{R}_+^{n+1},
$$
and the subharmonicity follows from the fact that $F(x,t)$ satisfies the generalized Cauchy-Riemann equations; see $\S 3$ in Chapter VII of \cite{St1} or $\S 4$ in Chapter III of \cite{St2} for details.

\medskip

(ii). Let $\mu_0$ be the finite measure in Lemma~\ref{lem:Riesz-CMO-2} and $\mu_j=R_j(d\mu_0)$ for $j=1,\ldots, n$. Let
$$
    u_j(x,t):=e^{-t\sqrt{\L}} (d\mu_j)(x)=\int_{\mathbb{R}^n} \mathcal{P}_t(x,y)\, d\mu_j(y),   \   \     i=0,1,\ldots, n,
$$
be the Poisson-Stieltjes integral of the finite Borel measure $\mu_j$. By using estimates for the Poisson kernels associated to  the semigroup $e^{-t\sqrt{\L}}$ given in \cite[Lemma~2.6]{SW},
it's clear that  each $u_j$ is  continuous  in $\mathbb{R}_+^{n+1}$  and 
$$
   \sup_{t>0}  \int_{\mathbb{R}^n}  |u_j(x,t)|\, dx\leq C\,  |\mu_j|(\mathbb{R}^n)<\infty.
$$
Obviously,  $u_j$ is an $\mathbb{L}-$harmonic function associated to the operator $\mathbb{L}=-\partial_{tt}+\L$ in the sense of 
$$
     \int_{\mathbb{R}_+^{n+1}} \nabla u_j\cdot \nabla \psi \, dY+\int_{\mathbb{R}_+^{n+1}}  V u_j\,\psi \, dY=0,\quad  \forall \, \psi\in C_0^1(\mathbb{R}_+^{n+1}),
$$
where $\nabla=(\nabla_x, \,\partial_t)$, and the  capital letter $Y=(y,t)$  denotes a point in $\mathbb R_+^{n+1}$

Moreover, we now give an extension of Lemma~2.6 in \cite{DYZ} that the index $p\geq 1$ therein can be extended to $p>0$:
for any $B(Y,4r)\subseteq \mathbb{R}_+^{n+1}$,
\begin{equation}\label{eqn:subharmonic}
    \sup_{(x,t)\in B(Y,r/2)} |u_j(x,t)|^p \leq  \frac{ c_p}{|B(Y,r)|} \int_{B(Y,r)}|u_j(x,t)|^{p} dx \,dt   \quad {\rm for}\  \  p>0.
\end{equation}
To this end, we claim that 
\begin{itemize}
	\item [\textbf{(F)}] for each $j=0,1,\ldots, n$,  $|u_j (x,t)|^2$ is a non-negative sub-harmonic function in $\mathbb{R}_+^{n+1}$. 
\end{itemize}

\smallskip

Let ${\rm Re\,}z$ and $\overline{z}$ be the real part and the complex conjugate  of $z\in \mathbb{C}$, respectively. Let $\langle \mathbf{z}, \mathbf{w}\rangle =\sum_{j=1}^{n+1} z_j\, \overline{w_j}$ for $\mathbf{z}=(z_1,\ldots, z_{n+1}),\,   \mathbf{w}=(w_1,\ldots, w_{n+1})\in \mathbb{C}^{n+1}$.

For any given $Y\in \mathbb{R}_+^{n+1}$ and $B=B(Y,4r)\subseteq \mathbb{R}_+^{n+1}$,  let $\varphi\geq 0$ be a Lipschitz function satisfying ${\rm supp}\, \varphi\subseteq B$, we have
\begin{align*}
	\iint_{B} \langle \nabla_{x,t}\, |u_j|^2 , \nabla _{x,t} \,\varphi \rangle\, dxdt &= 2 \iint_{B}   \langle    {\rm Re\,} \big( \overline{u_j} \,\nabla_{x,t}\, u _j\big), \nabla _{x,t} \,\varphi\rangle   \,  dx\,dt\\
	&=
	  2{\rm Re\,}  \iint_{B}   \langle \nabla_{x,t}\, u_j , \nabla _{x,t} \big(u_j\varphi\big) \rangle\, dx\,dt -2 {\rm Re\,} \iint_{B}     \langle \nabla_{x,t}\, u_j , \varphi\nabla _{x,t} \, u_j \rangle\, d x\,dt\\
	&=  -2{\rm Re\,} \iint_{B}   \big(\Delta_{x,t}\, u_j\big)  \overline{u_j} \,\varphi \, dx\,dt  -2 \iint_B |\nabla_{x,t}u_j|^2 \varphi \, dx\,dt\\
	&= -2\iint_B V |u_j|^2  \varphi \ dx\,dt -2 \iint_B |\nabla_{x,t}u_j|^2 \varphi \, dx\,dt\\
	&\leq 0.
\end{align*}
Hence $|u_j(x,t)|^2$ is weakly subharmonic, and so the claim \textbf{(F)} holds by Problem~2.8 in \cite[p. 29]{GT}. This allows us to apply Theorem~5.4 in \cite{BM} to see for every $p>0$, 
$$
    \sup_{(x,t)\in B(Y,r/2)} |u_j(x,t)|^2 \leq c_p \left( \frac{1}{|B(Y,r)|} \int_{B(Y,r)}|u_j(x,t)|^{2p} dx\, dt \right)^{1/p},
$$
where $c_p<\infty$ is a positive constant depending on $p$. As a consequence,  \eqref{eqn:subharmonic} follows readily. Indeed, one may verify that the function $u_j$ in \eqref{eqn:subharmonic} can be replaced by any $\mathbb{L}-$harmonic function in the ball $B(Y,4r)$.

 Furthermore, let 
$$
     F_{\L}(x,t)=\big(u_0(x,t),u_1(x,t),\ldots, u_n(x,t)\big)
$$
and $|F_{\L}(x,t)|^2=\sum_{j=0}^n |u_j(x,t)|^2$. The argument above shows that $|F_{\L}(x,t)|^2$ is a non-negative sub-harmonic function in $\mathbb{R}_+^{n+1}$ and 
$$
    \sup_{(x,t)\in B(Y,r/2)} |F_{\L}(x,t)|^p \leq  \frac{ c_p}{|B(Y,r)|} \int_{B(Y,r)}|F_{\L}(x,t)|^{p} dx\, dt   \quad {\rm for}\  \  p>0.
$$

  It's natural to ask whether or not we can establish the subharmonicity of $|F_{\L}|^p$ for some $p\leq 1$, by noticing the generalized Cauchy--Riemann equations are now no longer satisfied. Furthermore, it's interesting to  consider the possibility of 
establishing an analogous version of the F. and M. Riesz theorem associated to $\L$ such that the finite measure $\mu_0$ in Lemma~\ref{lem:Riesz-CMO-2}  must be absolutely continuous with Radon-Nikodym derivative in $H_{\L}^1(\mathbb{R}^n)$, that is, there exists $f\in H_{\L}^1(\mathbb{R}^n)$ such that $d\mu_0(x)=f(x)\,dx$.
\end{remark}

 \medskip

\section{An approximation to the identity and ${\rm CMO}_{\L}(\mathbb R^n)$: proof of Theorem  \ref{thm:approx}} \label{sec:approx}
\setcounter{equation}{0}

In the end, we turn to consider an approximation to the identity
 arising from the semigroups associated to $\L$. 
 
 Actually, this is not a trivial fact, since the standard approximation to the identity can not match ${\rm CMO}_{\L}(\mathbb R^n)$ well due to the potential $V$. 
 Even for a radial bump function $\phi$ satisfying
$$ 
   {\rm supp}\, \phi\subseteq B(0,1),\quad        0\leq \phi\leq 1  \quad {\rm and }\quad \int \phi(x)\, dx=1,
$$ 
 the convolution $A_tf=t^{-n}\phi(\cdot/t)*f $ for $f\in {\rm CMO}_{\L}(\mathbb R^n)$ satisfies  $\widetilde{\gamma}_1(A_t (f))=\widetilde{\gamma}_2(A_t (f))=\widetilde{\gamma}_3(A_t (f))=\widetilde{\gamma}_4(A_t (f))=0$, while the remaining $\widetilde{\gamma}_5(A_t (f))=0$ needs furthermore conditions on $f$ such as  compact support; see \cite[Lemma 4.1]{SW}. This means that the usual average of  a ${\rm CMO}_{\L}$ function may not fall into ${\rm CMO}_{\L}$, which is quite different from the standard identity approximation in the classical ${\rm CMO}$ space and the ${\rm CMO}_{-\Delta + 1}$ space (see \cite{D}).

However, we will see that the limit behavior of the Poisson integral of $f\in {\rm CMO}_{\L}$ also possesses nice approximate properties, i.e., Theorem  \ref{thm:approx}. The argument is also workable for the heat semigroups.

To show Theorem  \ref{thm:approx}, we introduce the following auxiliary result first.

\begin{lemma}\label{lem:Poisson-CMO}
Suppose $V\in RH_q$ for some $q\geq n/2$ and  let $\L=-\Delta+V$.  
There exists a constant $C>0$ such that 
  for any given $s>0$ and $f\in {\rm BMO}_{\L}(\mathbb R^n)$, we have $e^{-s\sqrt{\L }}  f\in {\rm BMO}_{\L}(\mathbb R^n)$, and 
\begin{equation}\label{eqn:semigroup-1}
     \left\|e^{-s \sqrt{\L}} f\right\|_{{\rm BMO}_{\L}(\mathbb R^n)}\leq C 
     \|f\|_{{\rm BMO}_{\L}(\mathbb R^n)}.
\end{equation}
 
Additionally, if $f\in {\rm CMO}_{\L}(\mathbb R^n)$, then  $e^{-s\sqrt{\L} }  f$ belongs to ${\rm CMO}_{\L}(\mathbb R^n)$ as well.
\end{lemma}

For any given $s>0$, it's clear that $\left|e^{-s\sqrt{\L}}f(x)\right|\leq C Mf(x)$. However, this, combined with Theorem  \ref{thm:M-CMO},  can not be used to deduce the ${\rm BMO}_{\L}(\mathbb R^n)$ norm of $e^{-s\sqrt{\L}}f$.
To prove Lemma \ref{lem:Poisson-CMO}, we apply the characterization of ${\rm CMO}_{\L}(\mathbb R^n)$
  in terms of tent spaces.

\begin{theorem}\label{thm:tent-CMO}
(see \cite[Theorem B]{SW}) 
Suppose $V\in {\rm RH}_q$ for some $q\geq  n/2$.
Then $f\in {\rm CMO}_{\L}$ if and only if  $f\in L^2(\mathbb{R}^n, (1+|x|)^{-(n+\beta)}dx)$ for some $\beta>0$ and $t\sqrt{\L}e^{-t\sqrt{\L}} f\in T_{2,C}^\infty$, with
$$
    \|f\|_{{\rm BMO}_{\L}}\approx \left\|  t\sqrt{\L}e^{-t\sqrt{\L}}f \right\|_{T_2^\infty}.
$$
\end{theorem}

The space $T_2^\infty$ is the class of functions $F$  defined on $\mathbb{R}_+^{n+1}$ for which $\mathfrak{C}(F)\in L^\infty(\mathbb{R}^n)$ and the norm $\|F\|_{T_2^\infty}=\|\mathfrak{C}(F)\|_{L^{\infty}}$, where  
$$
    \mathfrak{C}(F)(x)=\sup_{x\in B} \left(r_B^{-n} \iint_{\widehat{B}} |F(y,t)|^2 \frac{dy\, dt}{t} \right)^{1/2}.
$$
It's well known  from the Carleson measure that $f\in {\rm BMO}_{\L}(\mathbb R^n)$ if and only if $t\sqrt{\L}e^{-t\sqrt{\L}} f\in T_{2}^\infty$.
Moreover, we say $F\in T_{2,C}^\infty$  if $F\in T_{2}^\infty$ and 
\begin{itemize}
	\item [(i)] $\displaystyle
    \eta_1(F):=\lim_{a\to 0}   \sup_{B:\, r_B\leq a} \left(r_B^{-n} \iint_{\widehat{B}} |F(y,t)|^2  \frac{dy\, dt}{t}\right)^{1/2}=0$,
	\smallskip
	\item [(ii)] $\displaystyle \eta_2(F):=\lim_{a\to +\infty}  \sup_{B:\, r_B\geq a} \left(r_B^{-n} \iint_{\widehat{B}} |F(y,t)|^2  \frac{dy\, dt}{t}\right)^{1/2}=0$,
	\smallskip
	\item [(iii)] $\displaystyle \eta_3(F):=\lim_{a\to +\infty}  \sup_{B:\, B\subseteq \left(B(0,a)\right)^c} \left(r_B^{-n} \iint_{\widehat{B}} |F(y,t)|^2  \frac{dy\, dt}{t}\right)^{1/2}=0$,
\end{itemize}
where $\widehat {B}$ is the classical tent of $B$. Clearly, one may replace $\widehat {B}$ by $B\times (0,r_B)$, and 
by a similar argument, one may also characterize ${\rm CMO}_{\L}(\mathbb R^n)$ in terms of the heat semigroup of $\L $ rather than its Poisson counterpart. That is, the condition $t\sqrt{\L}e^{-t\sqrt{\L}} f\in T_{2,C}^\infty$ involved in Theorem \ref{thm:tent-CMO} can be replaced by $F'(y,t):=t^2\L e^{-t^2\L} f\in T_{2,C}^\infty$.  This observation implies that one may verify $e^{-s\L}f\in {\rm CMO}_{\L}$ for any fixed $s>0$
in a similar manner.

\smallskip

\begin{proof}[Proof of  Lemma \ref{lem:Poisson-CMO}]
For any fixed $s>0$,
 let 
$$F_s(y,t):=t\sqrt{\L}e^{-t\sqrt{\L}}e^{-s\sqrt{\L}}(y).
$$

{\it  Step I.}
we claim that $F_s\in T_2^\infty$, that is, $\mathfrak{C}(F_s)\in L^\infty$.

To see it, for any $x\in \mathbb R^n$  and for any ball $B=B(x_B,r_B)$ containing $x$, if $s\leq r_B$, then
\begin{align}\label{eqn:tent-1}
   \left( r_B^{-n} \int_{0}^{r_B}\int_{B} |F_s(y,t)|^2 \frac{dy\, dt}{t}\right)^{1/2}&= \left(  r_B^{-n}\int_0^{r_B}\int_B t|e^{-(t+s)\sqrt{\L}}f(y)|^2 dydt \right)^{1/2}\nonumber\\
    &\leq C  \left( (2r_B)^{-n}\int_0^{2r_B} \int_{2B} |\tau e^{-\tau\sqrt{\L}}f(y)|^2 \frac{dyd\tau}{\tau} \right)^{1/2}\nonumber\\
    &\leq  C\mathfrak{C}\left(t\sqrt{\L}e^{-t\sqrt{\L}}f\right)(x).
\end{align}
Otherwise, $s\geq r_B$, then for any $y\in B$ and $0<t<r_B$,
\begin{align}\label{eqn:tent-2}
   |F_s(y,t)|& =\left|e^{-s\sqrt{\L}}\left(t\sqrt{\L}e^{-t\sqrt{\L}}f\right)(y)\right|\nonumber\\
   &=\left| \left\{ \int_{B(y,s)} +  \sum_{k=1}^\infty \int_{B(y,2^ks)\setminus B(y,2^{k-1}s)} \right\}K_{e^{-\sqrt{\L}}}(y,z) t\sqrt{\L}e^{-t\sqrt{\L}}f (z)dz\right|\nonumber\\
   &\leq C  \sum_{k=0}^\infty \frac{1}{2^k} \left|t\sqrt{\L}e^{-t\sqrt{\L}}f\right|_{B(y,2^ks)}\nonumber\\
   &\leq   C  \sum_{k=0}^\infty \frac{1}{2^k} \left|t\sqrt{\L}e^{-t\sqrt{\L}}f\right|_{B(x,2^{k+1}s)},
\end{align}
then
\begin{align}\label{eqn:tent-3}
&\left( r_B^{-n} \int_{0}^{r_B}\int_{B} |F_s(y,t)|^2 \frac{dy\, dt}{t}\right)^{1/2}\nonumber\\
\lesssim &\sum_{k=0}^\infty  \frac{1}{2^k}  \left( \int_0^{r_B}  \left|t\sqrt{\L}e^{-t\sqrt{\L}}f\right|_{B(x,2^{k+1}s)}^2  \frac{dt}{t}\right)^{1/2}\nonumber\\
\lesssim & \sum_{k=0}^\infty  \frac{1}{2^k}  \left(  \int_0^{r_B} \frac{1}{|B(x,2^{k+1}s)|}\int_{B(x,2^{k+1}s)}\left|t\sqrt{\L}e^{-t\sqrt{\L}}f(z)\right|^2  \frac{dzdt}{t}\right)^{1/2}  \nonumber\\
\leq&\mathfrak{C}\left(t\sqrt{\L}e^{-t\sqrt{\L}}f\right)(x).
\end{align}

Hence, for any $x\in \mathbb R^n$, 
$$
   \mathfrak{C}(F_s)(x)\leq C  \mathfrak{C}\left(t\sqrt{\L}e^{-t\sqrt{\L}}f\right)(x)
$$
and the constant $C>0$ is independent of $x$ and $s>0$.

Due to the characterization of ${\rm BMO}_{\L}(\mathbb R^n)$ via the Carleson measure, we obtain \eqref{eqn:semigroup-1}.

\smallskip

{\it  Step II.} We continue to verify that $\eta_i(F_s)=0$ for $i=1,2,3$, which ensures $e^{-s\sqrt{\L} }  f\in {\rm CMO}_{\L}(\mathbb R^n)$.

For any given $\varepsilon>0$, by $\eta_i(t\sqrt{\L}e^{-t\sqrt{\L}}f)=0$ for $i=1,2,3$, there exist two integers $\mathcal I_{\varepsilon}>>1$ and $\mathcal J_{\varepsilon}>>1$ such that 
 \begin{subequations}
\begin{equation}\label{eqn:Tent-a}
   \sup_{B:\, r_B\leq 2^{-\mathcal I_\varepsilon}} \left(r_B^{-n} \int_0^{r_B}\int_{B} |t\sqrt{\L}e^{-t\sqrt{\L}}f(y)|^2  \frac{dy\, dt}{t}\right)^{1/2}< \varepsilon,
\end{equation}
\begin{equation}\label{eqn:Tent-b}
    \sup_{B:\, r_B\geq  2^{\mathcal J_\varepsilon}}  \left(r_B^{-n} \int_0^{r_B}\int_{B} |t\sqrt{\L}e^{-t\sqrt{\L}}f(y)|^2  \frac{dy\, dt}{t}\right)^{1/2}< \varepsilon,
\end{equation}
\begin{equation}\label{eqn:Tent-c}
     \sup_{B:\, B\subseteq (B(0, 2^{\mathcal J_\varepsilon} ))^c  }  \left(r_B^{-n} \int_0^{r_B}\int_{B} |t\sqrt{\L}e^{-t\sqrt{\L}}f(y)|^2  \frac{dy\, dt}{t}\right)^{1/2}< \varepsilon.
\end{equation}
 \end{subequations}
 
Let's  consider $\eta_1(F_s)$.

For any  ball $B'=B(x_{B'},r_{B'})$ with $ r_{B'}<2^{-\mathcal I_{\varepsilon}-1}$ sufficiently small, 
 if $s\leq r_{B'}$, then  combine \eqref{eqn:tent-1}  and \eqref{eqn:Tent-a} to obtain
 $$
     \left( r_{B’}^{-n} \int_{0}^{r_{B‘}}\int_{B’} |F_s(y,t)|^2 \frac{dy\, dt}{t}\right)^{1/2}\leq C \varepsilon.
 $$
Otherwise, $r_{B'}< s$, then  apply \eqref{eqn:tent-2}  to see
\begin{align*}
   |F_s(y,t)| =\left|e^{-s\sqrt{\L}/2}\left(t\sqrt{\L}e^{-(t+s/2)\sqrt{\L}}f\right)(y)\right|\leq    C  \sum_{k=0}^\infty \frac{1}{2^k} \left|t\sqrt{\L}e^{-(t+s/2)\sqrt{\L}}f\right|_{B(x_{B'},2^k s)},
\end{align*}
and so
\begin{align*}
&\left( r_{B’}^{-n} \int_{0}^{r_{B‘}}\int_{B’} |F_s(y,t)|^2 \frac{dy\, dt}{t}\right)^{1/2}\\
\leq & C \sum_{k=0}^\infty  \frac{1}{2^k}  \left(  \int_0^{r_B} \frac{1}{|B(x_{B'},2^{k+1}s)|}\int_{B(x_{B'},2^{k+1}s)}   \frac{t}{t+s/2}(t+s/2)\left|\sqrt{\L}e^{-(t+s/2)\sqrt{\L}}f(z)\right|^2   dzdt \right)^{1/2}  \\
\leq & C \sqrt{\frac{r_{B'}}{s}}  \sum_{k=0}^\infty  \frac{1}{2^k}  \left(  \int_0^{2^{k+1}s} \frac{1}{|B(x_{B'},2^{k+1}s)|}\int_{B(x_{B'},2^{k+1}s)}    \left|\tau\sqrt{\L}e^{-\tau\sqrt{\L}}f(z)\right|^2   \frac{dzd\tau}{\tau} \right)^{1/2} \\
\leq & C \sqrt{\frac{r_{B'}}{s}}   \left\|t\sqrt{\L}e^{-t\sqrt{\L}}f\right \|_{T_2^\infty}\\
\leq & C\sqrt{\frac{r_{B'}}{s}}    \|f\|_{{\rm BMO}_{\L}(\mathbb R^n)}.
\end{align*}
Note that $s>0$ is fixed, thus 
\begin{equation}\label{eqn:small}
    \sup_{B':\, r_{B'}\leq s\varepsilon^2}   \left( r_{B’}^{-n} \int_{0}^{r_{B‘}}\int_{B’} |F_s(y,t)|^2 \frac{dy\, dt}{t}\right)^{1/2}\leq C\|f\|_{{\rm BMO}_{\L}(\mathbb R^n)}\varepsilon.
\end{equation}
Consequently, $\eta_1(F_s)=0$ from these two cases.
 
 \smallskip
 
 To continue, we consider $\eta_2(F_s)$. 
  For any  ball $B'=B(x_{B'},r_{B'})$ with $ r_{B'}
  \geq 2^{\mathcal J_{\varepsilon}}$ sufficiently large, if $s\leq r_{B'}$, then  combine \eqref{eqn:tent-1}  and \eqref{eqn:Tent-b} to obtain
 $$
     \left( r_{B’}^{-n} \int_{0}^{r_{B‘}}\int_{B’} |F_s(y,t)|^2 \frac{dy\, dt}{t}\right)^{1/2}\leq C \varepsilon
 $$
as well. If   $s> r_{B'}$, then $s>2^{\mathcal J_{\varepsilon}}$, and it follows from \eqref{eqn:tent-2},  \eqref{eqn:tent-3} and \eqref{eqn:Tent-b} to see
\begin{align*}
  \left( r_{B'}^{-n} \int_{0}^{r_{B'}}\int_{B'} |F_s(y,t)|^2 \frac{dy\, dt}{t}\right)^{1/2}\leq C \sum_{k=0}^\infty  \frac{1}{2^k}  \varepsilon\leq C\varepsilon.
\end{align*}
 Thus $\eta_2(F_2)=0$.

It remains to consider $\eta_3(F_s)$.  For any  ball $B'=B(x_{B'},r_{B'})$ which is far away from the origin, it suffices to assume that $r_{B'}<2^{\mathcal J_\epsilon}$ due to the argument of $\eta_2(F_s)=0$. Furthermore, assume that  $B'\subseteq  (B(0, 2^{\mathcal J_\varepsilon +1} ))^c$, then $2B'\subseteq  (B(0, 2^{\mathcal J_\varepsilon} ))^c$. This, combined with  \eqref{eqn:tent-1}  and \eqref{eqn:Tent-c}, implies that %if $s\leq r_{B'}$,
$$
   \left( r_{B’}^{-n} \int_{0}^{r_{B‘}}\int_{B’} |F_s(y,t)|^2 \frac{dy\, dt}{t}\right)^{1/2}\leq C \varepsilon  \   \   \text{if}\  \  s\leq r_{B'}.
$$

Otherwise, if $s>r_{B'}$, then 
\begin{align*}
  & \left( r_{B’}^{-n} \int_{0}^{r_{B‘}}\int_{B’} |F_s(y,t)|^2 \frac{dy\, dt}{t}\right)^{1/2}\\
      \leq & C   \sum_{k=0}^\infty  \frac{1}{2^k}  \left(  \int_0^{2^{k+1}s} \frac{1}{|B(x_{B'},2^{k+1}s)|}\int_{B(x_{B'},2^{k+1}s)}    \left|\tau\sqrt{\L}e^{-\tau\sqrt{\L}}f(z)\right|^2   \frac{dzd\tau}{\tau} \right)^{1/2}.
\end{align*}
Using \eqref{eqn:Tent-b} again, it suffices to consider the case $r_{B'}<s<2^{\mathcal J_\epsilon}$.

let $N_{\varepsilon}\in \mathbb N_+$ such that $\sum_{k=N_{\varepsilon}+1}^\infty  2^{-k} <\varepsilon$, then whenever $B'\subseteq (B(0,2^{\mathcal J_{\varepsilon}  +N_{\varepsilon}+1}))^c$,  it's clear that $B(x_{B'},2^{k+1}s)\subseteq (B(0,2^{\mathcal J_{\varepsilon}  }))^c$ for $0\leq k\leq N_{\varepsilon}$.
Hence  we use \eqref{eqn:Tent-c} to see
\begin{align*}
      \left( r_{B’}^{-n} \int_{0}^{r_{B‘}}\int_{B’} |F_s(y,t)|^2 \frac{dy\, dt}{t}\right)^{1/2}
         \leq & C\sum_{k=0}^{N_{\varepsilon}} \frac{1}{2^k}\varepsilon  +\sum_{k=N_{\varepsilon} +1}^\infty \frac{1}{2^k} \|f\|_{{\rm BMO}_{\L}}\\
      \leq & C(\|f\|_{{\rm BMO}_{\L}}+1) \varepsilon.
\end{align*}
Hence $\eta_3(F_s)=0$. We complete the proof of  Lemma \ref{lem:Poisson-CMO}.
\end{proof}
%%%%%%%%%%%%%%%%

\bigskip

\begin{remark}
Note that the constant $C$ in \eqref{eqn:tent-1} and \eqref{eqn:tent-2} is independent of $s>0$, hence 
\begin{align*}
   \sup_{s>0}\left\|e^{-s\sqrt{\L}}f\right\|_{{\rm BMO}_{\L}}&\approx   \sup_{s>0} \left\|t\sqrt{\L}e^{-t\sqrt{\L}}(e^{-s\sqrt{\L}}f)\right\|_{T_2^\infty}\\
   &\leq C \left\|t\sqrt{\L}e^{-t\sqrt{\L}} f\right\|_{T_2^\infty}\approx \|f\|_{{\rm BMO}_{\L}}.
\end{align*}

\end{remark}

\begin{remark}
It's natural to continue to study the behavior of  the maximal operator $\mathcal P^*$ defined by
$$
   \mathcal P^* f(x)=\sup_{s>0} \left|e^{-s}f(x)\right|,
$$
on ${\rm CMO}_{\L}(\mathbb R^n)$. Recall that 
 it has been shown in \cite{DGMTZ} that $\mathcal P^*$ is bounded on ${\rm BMO}_{\L}(\mathbb R^n)$. On one hand,  this result cannot deduce our \eqref{eqn:semigroup-1}. On the other hand, the condition $r_{B'}\leq s\varepsilon^2$ in \eqref{eqn:small} hints that it seems hard to estimate  the limit behaviour $\eta_1(\mathcal P^*f)$ trivially.

\end{remark}

\bigskip

Based on Lemma \ref{lem:Poisson-CMO}, we continue to finish the remaining argument of Theorem \ref{thm:approx}.

\begin{proof}[Proof of Theorem \ref{thm:approx}]
For any $f\in {\rm CMO}_{\L}(\mathbb R^n)$, note that ${\rm CMO}_{\L}(\mathbb R^n)$ is the closure of $C_{c}^\infty(\mathbb R^n)$ in ${\rm BMO}_{\L}(\mathbb R^n)$, hence there exists a sequence $\{f_k\}_k $ in $C_{c}^\infty(\mathbb R^n)$ such that 
$$
    \lim_{k\to \infty}\left\| f_k-f\right\|_{{\rm BMO}_{\L}(\mathbb R^n)}=0,
$$
this, combined with the (uniform) boundedness of $e^{-t\sqrt{\L}}$ on ${\rm BMO}_{\L}(\mathbb R^n)$ for $t>0$, deduces that for any $k\in \mathbb N_+$, 
\begin{align*}
\left\|e^{-t\sqrt{\L}}f-f \right\|_{{\rm BMO}_{\L}} &\leq \left\|e^{-t\sqrt{\L}}(f-f_k) \right\|_{{\rm BMO}_{\L}} + \left\|f_k-f\right\|_{{\rm BMO}_{\L}} +  \left\|e^{-t\sqrt{\L}}f_k-f_k \right\|_{{\rm BMO}_{\L}}\\
&\leq C \left\|f_k-f\right\|_{{\rm BMO}_{\L}}+  \left\|e^{-t\sqrt{\L}}f_k-f_k \right\|_{{\rm BMO}_{\L}},
\end{align*}
where the positive constant $C$ is independent of $t$.

 Hence,  to prove \eqref{eqn:approx-identity-CMO}, it suffices to verify it for $f\in C_c^\infty(\mathbb R^n)$.  
 
 We start by showing that for any $f\in C_c^\infty(\mathbb R^n)$,
  $\displaystyle \lim_{t\to 0} e^{-t\sqrt{\L}}f(x)=f(x)$ uniformly for all $x\in \mathbb R^n$, which is crucial for our purpose.

Note that by the Kato--Trotter formula, 
there exists constants $C,\, c>0$ such that for all $x,y\in \mathbb R^n$ and $t>0$,
\begin{equation}\label{eqn:compare-heat}
0\leq K_{e^{t\Delta}}(x,y)- K_{e^{-t\L}} (x,y)\leq C  \left(\frac{\sqrt{t}}{\sqrt{t}+\rho(x)}\right)^{2-\frac{n}{q_1}} \frac{1}{t^{n\over 2}} \exp\left(-\frac{|x-y|^2}{ct}\right),
\end{equation}
where $q_1>q\geq n/2$ is the index in the proof of Theorem~\ref{thm:Riesz-CMO};
see \cite[Proposition 7.13]{BDL}. Let 
$$
\delta=\min\left\{2-\frac{n}{q_1}, \frac{1}{2}\right\},
$$
then combining the Bochner subordination formula
$$
e^{-t \sqrt{\L}} =\frac{1}{2 \sqrt{\pi}} \int_{0}^{\infty} \frac{t}{\sqrt{s}} \exp \left(-\frac{t^{2}}{4 s}\right) e^{-s \L} \frac{d s}{s},
$$
we have that whenever $t<\rho(x)$, 
\begin{align}\label{eqn:Poisson-compare}
0\leq& K_{e^{-t\sqrt{-\Delta}}}(x,y)- K_{e^{-t\sqrt{\L}}} (x,y) \nonumber\\
=& \frac{1}{2 \sqrt{\pi}} \int_{0}^{\infty} \frac{t}{\sqrt{s}} \exp \left(-\frac{t^{2}}{4 s}\right)\left[K_{e^{s\Delta}}(x,y)- K_{e^{-s\L}} (x,y)\right] \frac{d s}{s}\nonumber\\
\leq& C\int_0^{\infty}  \frac{t}{\sqrt{s}}    \left(\frac{\sqrt{s}}{\sqrt{s}+\rho(x)}\right)^{\delta} \frac{1}{\sqrt{s^n}} \exp\left(-\frac{t^2}{4s}-\frac{|x-y|^2}{cs}\right)        \frac{d s}{s}\nonumber\\
\leq& C   \left(\frac{t}{\rho(x)}\right)^{\delta}   \int_0^{t^2+|x-y|^2}   \left(\frac{t}{\sqrt{s}}\right)^{1-\delta} \frac{1}{s^{n\over 2}}  \frac{s^{\frac{n+1}{2}}}  { (t^2+|x-y|^2)^{
\frac{n+1}{2}}     } \frac{ds}{s}  \nonumber\\
& +C  \left(\frac{t}{\rho(x)}\right)^{\delta}   \int_{t^2+|x-y|^2}^{\infty}   \left(\frac{t}{\sqrt{s}}\right)^{1-\delta} \frac{1}{s^{n\over 2}}  \frac{ds}{s}  \nonumber\\
\leq & C  \left(\frac{t}{\rho(x)}\right)^{\delta}    \frac{t^{1-\delta}}{ (t^2+|x-y|^2)^{
\frac{n+1-\delta}{2}}   }.
\end{align}

For $f\in C_c^\infty(\mathbb R^n)$, for any $\varepsilon>0$,  
 it's clear that $e^{-t\sqrt{-\Delta}}f(x)\to f(x)$ uniformly for $x\in \mathbb R^n$ as $t\to 0$.  That is, for any $\varepsilon>0$, there exists some $t_0>0$ such that for any $t\leq t_0$,
$$
    \left|e^{-t\sqrt{-\Delta}}f(x)-f(x)\right|< 
    \varepsilon,\quad \forall\, x\in \mathbb R^n.
$$
Meanwhile, suppose ${\rm supp\,}f\subseteq B(0,M_1)$ for some $M_1>0$. Then there exists $M_2>M_1$ such that for any $t\leq t_0$, 
$$
    \left|  e^{-t\sqrt{\L}}f(x)\right|<\varepsilon  \quad \text{for}\  \   |x|\geq M_2.
$$
Let 
 $$
    \rho_{\min}=\inf_{x\in B(0,M_2)} \rho(x),
 $$
 then $\rho_{\min}>0$ by Lemma \ref{lem:size-rho}.
 Without loss of generality, assume that  $t_0$ satisfies 
 $$
     \left(\frac{\sqrt{t_0}}{\rho_{\min}}\right)^{\delta}< 
     \varepsilon.
 $$
 Then for any $t\leq t_0$ and $x\in \mathbb R^n$, one may combine \eqref{eqn:Poisson-compare}  and $f\in C_c^\infty(\mathbb R^n)$ to see
 \begin{itemize}
 \item if $|x|<M_2$, 
\begin{align*}
 \left| e^{-t\sqrt{\L}}f(x)-f(x)\right|&\leq   \left| e^{-t\sqrt{\L}}f(x)-e^{-t\sqrt{-\Delta}}f(x)\right| +\left| e^{-t\sqrt{-\Delta}}f(x)-f(x)\right|  \nonumber\\
  & \leq  CMf(x) \cdot  \varepsilon +     \varepsilon \nonumber\\
  &\lesssim  \varepsilon.
\end{align*}

\smallskip

\item if $|x|\geq M_2$, 
$$
    \left| e^{-t\sqrt{\L}}f(x)-f(x)\right|=\left| e^{-t\sqrt{\L}}f(x)\right|   
\leq   \varepsilon  
$$
 \end{itemize}
 Therefore,

 \begin{align}\label{eqn:Poisson-2}
   \left| e^{-t\sqrt{\L}}f(x)-f(x)\right|\lesssim \varepsilon \quad \text{for}\   \   t\leq t_0 \   \text{and}\         x\in \mathbb R^n.
 \end{align}
 Hence $e^{-t\sqrt{\L}}f\to f$ uniformly for all $x\in\mathbb R^n $ as $t\to 0$.
 
 It remains to prove \eqref{eqn:approx-identity-CMO} for $f\in C_c^\infty(\mathbb R^n)$.
 
 For any ball $B=B(x_B,r_B)$ and for any   $t<t_0$,
 \begin{itemize}
 \item if $r_B<\rho(x_B)$, then  by \eqref{eqn:Poisson-2},
 \begin{align*}
 \frac{1}{|B|}\int_B \left|  e^{-t\sqrt{\L}}f(x)-f(x) -\left(e^{-t\sqrt{\L}}f-f\right)_B\right|^2dx\lesssim \sup_{x\in B}   \left|  e^{-t\sqrt{\L}}f(x)-f(x)\right|\lesssim \varepsilon.
 \end{align*}

 \item if $r_B\geq \rho(x_B)$, similarly, by \eqref{eqn:Poisson-2} again, we also have 
   \begin{align*}
 \frac{1}{|B|}\int_B \left|  e^{-t\sqrt{\L}}f(x)-f(x)  \right|^2  dx\lesssim  \varepsilon.
 \end{align*}
 \end{itemize}
Therefore, we complete the proof of Theorem \ref{thm:approx}.
\end{proof}

\begin{remark}
Similarly, due to \eqref{eqn:compare-heat} and the fact that for any $f\in C_c^\infty(\mathbb R^n)$,  $e^{t\Delta}f(x)\to f(x)$ uniformly for all $x\in \mathbb R^n$ as $t\to 0$,
the approximation to the identity arising from the heat semigroup associated to $\L$ also matches ${\rm CMO}_{\L}(\mathbb R^n)$. That is, for any $f\in {\rm CMO}_{\L}(\mathbb R^n)$, we have 
$$
   \lim_{t\to 0} e^{-t\L}f=f\quad {\rm in }\  \   {\rm BMO}_{\L}(\mathbb R^n).
$$
In particular, if $f\in C_c^\infty(\mathbb R^n)$, then  we also have $\displaystyle \lim_{t\to 0} e^{-t\L}f(x)=f(x)$ uniformly for all $x\in \mathbb R^n$.
\end{remark}

\begin{remark}
Recall that 
 $$
      \lim_{t\to 0} e^{-t\sqrt{\L}}f=f\quad {\rm in }\  \   L^p(\mathbb R^n)
 $$
 for $1\leq p<\infty$; see \cite{BDY}.  Our Lemma \ref{lem:Poisson-CMO} and Theorem \ref{thm:approx} can be regarded as certain endpoint results. All of these will be useful in further  applications such as function spaces, density arguments, partial differential equations and so on.
\end{remark}

 \bigskip

%%%%%%%%%%%%%%%%%%%%%%%%%%%%%

 \noindent{\bf Acknowledgments.}
Li is supported by ARC DP220100285.
 Wu is  supported by  NNSF of China \# 12201002,  Anhui NSF of China \# 2208085QA03 and Excellent University Research and Innovation Team in Anhui Province \# 2024AH010002.

 \medskip

 %%%%%%%%%%%%%%%%%%%%%%%%%%%%%%%%%%%%%%%%%%%%%%%%%%%%%%

%%%%%%%%%%%%%%%%%%%%%%%%%%%%%%%%%%%%%%%%%%%%%%%%%%%%%%

\end{document}